\documentclass[12pt,reqno]{amsart}

\usepackage{bbm}
\usepackage[bb=boondox]{mathalfa}
\usepackage[margin=1in]{geometry}
\usepackage{amsthm, amsmath}
\usepackage{tikz, tikz-cd}
\usepackage{hyperref}
\usepackage{enumitem}
\usepackage{quiver}
\usepackage{breqn}
\usepackage{lipsum}
\usepackage{adjustbox}
\usepackage{subcaption}
\theoremstyle{definition}
\usepackage{mathtools}
\tikzcdset{scale cd/.style={every label/.append style={scale=#1},
    cells={nodes={scale=#1}}}}

\usetikzlibrary{calc}

\usepackage[english]{babel}
\newtheorem{theorem}{Theorem}

\numberwithin{equation}{section}
\newtheorem{lemma}[theorem]{Lemma}

\theoremstyle{definition}

\newtheorem{definition}[theorem]{Definition}

\newtheorem{remark}[theorem]{Remark}

\definecolor{urlcolor}{rgb}{0,.145,.698}
\definecolor{linkcolor}{rgb}{.01,0.01,0.51}
\definecolor{citecolor}{rgb}{.12,.54,.11}
    
\hypersetup{
      breaklinks=true,
      colorlinks=true,
      urlcolor=urlcolor,
      linkcolor=linkcolor,
      citecolor=citecolor,
      }
\begin{document}

\pagestyle{plain}

\begin{abstract}
    Generalized Cluster Algebras (GCA) are generalizations of Cluster Algebras (CA) with higher-order exchange relations. Previously, Chekhov-Shapiro conjectured that every GCA can be embedded into a CA. In this paper, we prove a modified version of this conjecture by providing a construction that realizes a given GCA as subquotient of some CA, as an algebra over the ground ring of the GCA via restriction of scalars. 
\end{abstract}
\title{Generalized Cluster Algebras are Subquotients of Cluster Algebras}
\author[]{Rolando Ramos \quad \quad David Whiting}
\maketitle

\tableofcontents

\section{Introduction}
A Cluster Algebra is a commutative algebra whose generators are determined by a specific combinatorial process, introduced in \cite{fomin2002cluster}. These generators are called \emph{cluster variables} and come in finite collections, called clusters. In the case of the so-called skew-symmetric cluster algebras, 
these clusters are associated with the vertices of quivers (directed graphs containing  neither loops nor two-loops). Local transformations of the quivers about vertices, called mutations, form new clusters of cluster variables. 

Under a mutation $x_i$, the mutated cluster variable is a Laurent polynomial with binomial numerator \[x'_i=\dfrac{P_{in}+P_{out}}{x_i}\] where $P_{in}=\prod_{k:\  k\to i} x_k$, $P_{out}=\prod_{\ell: i\to \ell} x_j$.  A \emph{seed} of the cluster algebra is the pair of a cluster and the associated quiver. A remarkable property of cluster algebras is that any cluster variable obtained by a sequence of mutations of a fixed initial cluster can be expressed as a Laurent polynomial in the initial cluster; this is called a \textit{Laurent phenomenon}. 

In \cite{PoissonCluster}, the authors introduced the Poisson bracket compatible with a (geometric) cluster algebra. Vice versa, given a rational manifold with a homogeneous quadratic Poisson structure, the transformations that preserve the Poisson bracket and some additional natural restrictions were described.

As a result, the authors recovered the mutation rule of quivers and the exchange relation of clusters as unique involution transformations of log-canonical bases satisfying certain additional restrictions. Their complete description of so-called normalized involutive canonical local data associated with a Poisson manifold suggested a more general family of transformations and, hence, more general algebras than cluster algebras.

These \emph{generalized cluster algebras (GCA)} were rigorously defined first in \cite{Techmuller} where the Laurent phenomenon for GCA was demonstrated. The numerators of the Laurent polynomials in the GCA mutation rule could have more than two terms.
All examples of a GCA considered in \cite{Techmuller} stem from the Teichm\"uller space of two-dimensional hyperbolic surfaces with orbifold points and allow an  embedding into the corresponding classical cluster algebras defined in \cite{fomin2002cluster}.

In \cite{Techmuller}, the authors Chekhov-Shapiro conjectured that every generalized cluster algebra could be embedded into some classical cluster algebra. In this paper, confirm this conjecture in the form of theorem \ref{thm: Main Theorem} below.

We assume that a generalized cluster algebra $\mathcal{A}^g$ is defined over 
the group ring $\mathbb{ZP}$ of a semi-field $\mathbb{P}$. 
We consider another semi-field $\widehat{\mathbb{P}}$ and an ideal $\mathcal{I}\subset\mathbb{Z}\widehat{\mathbb{P}}$ and construct 
a cluster algebra $\mathcal{A}$ over the group ring $\mathbb{Z}\widehat{\mathbb{P}}$ such that $\mathbb{Z}\widehat{\mathbb{P}}/\mathcal{I}$ contains $\mathbb{Z}{\mathbb{P}}$ as a subring and $\mathcal{A}\diagup \mathcal{I}^e$ contains $\mathcal{A}$ as a $\mathbb{ZP}$ subalgebra (where  $\mathcal{I}^e$ is the ideal extension of $\mathcal{I}$ into $\mathcal{A}$). Therefore, we state the following theorem, where the subquotient is in the category of $\mathbb{ZP}$ algebras:

\begin{theorem}\label{thm: Main Theorem}
    Any Generalized Cluster Algebra is a Subquotient of a Cluster Algebra.
\end{theorem}

\noindent \textbf{Acknowledgements.} The authors are supported NSF research grant DMS \#2100791. We would like to thank M. Shapiro for introducing this problem and  S. Fomin for his recommendation on using the term subquotient.

\section{Preliminaries}


We now describe a class of matrices and a transformation of this class which are fundamental to Cluster Algebra theory. For an $N\times (N+M)$ matrix $B\in \mathbb{Z}^{N\times (N+M)}$ with integer entries, we refer to the leftmost $N\times N$ square sub-matrix of $B$ as its \textit{principal part} and denote it by $\hbox{Principal}(B)\in \mathbb{Z}^{N\times N}$. Similarly, we refer to the rightmost $N\times M$ sub-matrix of $B$ as the \textit{slack part} and denote it by $\hbox{Slack}(B)\in \mathbb{Z}^{N\times M}$. The definition of the principal part and slack of a tall $(M+N)\times M$ matrix are similar, and the choice of wide $N \times (N+M)$ matrices is a convention.

A square matrix $B$ is called $\textit{skew-symmetric}$  if $B=-B^{T}$ so that it is equal to the negative of its transpose. We denote the set of $N\times N$ skew-symmetric matrices as $\hbox{Skew}(\mathbb{Z},n)$. More generally, a square matrix $B\in \mathbb{Z}^{N\times N}$ is called \textit{skew-symmetrizable} if there exists a diagonal matrix $D\in \mathbb{Z}^{N\times N}$ such that the matrix product $DB\in \mathbb{Z}^{N\times N}\in \hbox{Skew}(\mathbb{Z},n)$ is skew-symmetric. We call the diagonal matrix $D$ the \textit{diagonalizer} of $B$. More generally, a rectangular matrix $B\in \mathbb{Z}^{N\times (N+M)}$ is called an \textit{extended exchange matrix} if its principal part is $\textit{skew-symmetrizable}$.

For $1 \leq i \leq N$, fix $d_i$ to be a divisor of $\hbox{gcd}(B_{i,j}: 1\leq j \leq N)$ so that each entry of the $i^{th}$ row of $\hbox{Principal}(B)$ is divisible by $d_i$. 
\begin{definition}
    Associated with a choice of divisors $\{d_i\}_{i=1}^N$ there is a matrix $\widehat{B}\in \mathbb{Z}^{N\times (N+M)}$ called the \textit{Modified Exchange Matrix} given by the following matrix partition,
\begin{equation}
    \widehat{B} = [\widehat{D}\cdot \hbox{Principal}(B)|\hbox{Slack}(B)]
\end{equation}
where $\widehat{D}$ is a diagonal matrix such that $ \widehat{D}= \hbox{diag}(\frac{1}{d_1},\frac{1}{d_2},...,\frac{1}{d_N})$. 
\end{definition}
In other words, the modified exchange matrix is constructed by dividing the rows of the principal part of the exchange matrix by the chosen set of divisors; and the new remaining entries in each row will remain integers after this division process. 

We now provide the standard definitions of a generalized cluster structure using the definitions from \cite{gekhtman2022generalized} (although with slight modifications of some definitions which we specify)\color{black}

In what follows, let $\mathbb{P}:=\hbox{trop}(X_{N+1},...,X_{N+M})$ be the set of Laurent monomials in the variables $X_{N+j}$. Now, by multiplying Laurent monomials as rational functions, this set can be given the structure of an abelian group $(\mathbb{P},\otimes)$. Essentially $(\mathbb{P},\otimes)$ is a free group on $M$ generators. Moreover, by the following binary operation,
\begin{equation*}
    \prod_{j} (X_{N+j})^{\alpha_j} \oplus \prod_{j} (X_{N+j})^{\beta_j}=\prod_{j} (X_{N+j})^{\hbox{max}(\alpha_j,\beta_j)}
\end{equation*}
we can give the underlying abelian group the structure of a \textit{semi-field} as $(\mathbb{P},\oplus, \otimes)$. This particular semi-field is referred to as the \textit{tropical semi-field}.
\color{black}

\begin{definition}[Generalized seed (of geometric type)]\label{def: generalized seed}
    Let $\mathcal{F}$ be the field of rational functions in $N+M$ independent variables. We fix $M$ so called \textit{stable} variables $x_{N+1}, \cdots, x_{N+M}$ and consider the semifield $\mathbb{P}=\hbox{trop}(X_{N+1},...,X_{X+M})$. A \textit{generalized seed}  of rank n in $\mathcal{F}$ is a triple $\textbf{t} = (\text{\textbf{x}}, B, \mathcal{P})$, where
    \begin{enumerate}
        \item \textbf{x} = $(x_1, \cdots, x_N)$ is a transcendence basis of $\mathcal{F}$ where $\mathcal{F}$ contains the field of rational functions $\mathbb{Q}(x_{N+1},...,x_{N+M})$ 
        \item $B$ is an extended exchange matrix,
        \item  $\mathcal{P}$ is a set of $N$ \textit{strings}, where the $i$th string is a tuple $(p_{i0}, \cdots, p_{id_i})$ of laurent monomials $p_{ij} \in \mathbb{P}$ such that $p_{i0} = p_{id_i} = 1$.\color{black}
    \end{enumerate}
    We call the $N$-tuple \textbf{x} $= (x_1, \cdots, x_N)$ a \textit{cluster} and the elements $x_1, \cdots, x_N$ \textit{cluster variables}. Furthermore, we call the monomials $p_{i,j}$ \textit{exchange coefficients}.
\end{definition}

\begin{remark}
    We have modified the string of coefficients $\mathcal{P}$ from \cite{gekhtman2022generalized} so that it now consists of Laurent monomials rather than just monomials. These exchange patterns still satisfy the caterpillar lemma, and so the associated algebra has the Laurent phenomena.
\end{remark} 
\color{black}

\begin{definition}\label{def:seed_mutation}
Given a generalized seed $\textbf{t}=(\text{\textbf{x}}, B, \mathcal{P})$ and an integer $k\in \{1,...,N\}$,  the \textit{adjacent seed in direction} $k$ is the seed $\mu_k(\textbf{t})= (\mu_k (\text{\textbf{x}}),\mu_k (B), \mu_k (P))$ whose components are given by 
   \begin{enumerate}
        \item $\mu_k(\text{\textbf{x}}) = (x_1, \cdots, x_{k-1}, x_k', x_{k+1}, \cdots, x_{N})$, where the new cluster variable $x_k'$ is given by the \textit{generalized exchange relation}
        \begin{align}\label{eq: Generalized Exchange Relation}
            x_kx_k' = \sum_{r=0}^{d_k} p_{kr} u^{r}_{k>}v^{[r]}_{k>} u^{d_k-r}_{k<}v^{[d_k-r]}_{k<},
        \end{align}
        where
        \begin{align*}u_{k>}:= \prod_{\substack{ 1\leq i\leq N \\ \widehat{B_{ki}} > 0}}x_{i}^{\widehat{B_{ki}}}, & \quad \quad v_{k>}^{[r]}:= \prod_{\substack{ N+1\leq i\leq N+M \\ \widehat{B_{ki}} > 0}}x_{i}^{\lfloor r\widehat{B_{ki}}/d_k \rfloor}, \\ u_{k<} := \prod_{\substack{ 1\leq i\leq N \\ \widehat{B_{ki}} < 0}}x_{i}^{-\widehat{B_{ki}}}, & \quad \quad v_{k<}^{[r]} := \prod_{\substack{ N+1\leq i\leq N+M \\ \widehat{B_{ki}} < 0}}x_{i}^{\lfloor -r\widehat{B_{ki}}/d_k \rfloor},   
        \end{align*}
        here [r] denotes $[r]:=\lfloor \frac{ rb_{xk}}{d_x} \rfloor$ with $\lfloor ...\rfloor$ being the integer floor function so that $[r]\in \mathbb{Z}$ is an integer.

       \item    \begin{equation}
                (\mu_k (B))_{i,j} = \begin{cases} -B_{i,j} & \text{$i = k$ or $j=k$} \\ B_{i,j}+ \frac{|B_{i,k}|B_{k,j}+B_{i,k}|B_{k,j}|}{2} & \hbox{otherwise} \end{cases}
                \end{equation}

       \item $\mu_k(\mathcal{P})$ is given by the strings $(p'_{i0}, \cdots, p'_{i_{d_i}})$, where
       \begin{align}\label{eq: p mutation rule}
           \mu_k (p_{ij}) = \begin{cases}
               p_{i,d_i-j}, & \text{if $i=k$}; \\
               p_{ij}, & \text{otherwise}.
           \end{cases}
       \end{align}
   \end{enumerate}
    We call the right-hand side of the generalized exchange relation the \textit{generalized exchange polynomial}, and it is denoted by $\theta_{k}$. We refer to the  monomials $\mu_{k>}$ and $\mu_{k<}$ as the \textit{cluster monomials} and the monomials $v_{k>}$ and $v_{k<}$ as the \textit{stable monomials} in direction $k$. We call the operation $(\text{\textbf{x}}, Q, \mathcal{P}) \mapsto (\mu_k(\text{\textbf{x}}), \mu_k(Q), \mu_k(\mathcal{P}))$ the \textit{mutation in direction $k$}, or in other words in the direction of the $k^{th}$ cluster variable.
\end{definition}

Given a frozen (stable) variable $x_j\in \mathbb{P}$, a (generalized) seed $\mathbf{t}$, a cluster variable $x_k^{\mathbf{t}}$, and natural number $r\in \mathbb{N}$, we define the following Laurent monomials in $\mathbb{P}$:
\begin{equation}\label{eq: frozen box}
    x_j^{[r]}:= (x_j)^{\lfloor r|\widehat{B}_{kj}| / d_k \rfloor} 
\end{equation}
The given cluster variable will be obvious from the context of the discussion every time we use this Laurent monomial in this work. Note that by this definition we have 
\begin{equation*}
    v_{k>}^{[r]}  = \prod_{\substack{ N+1\leq i\leq N+M \\ \widehat{B_{ki}} > 0}}  f_{i}^{[r]} \hbox{ and } v_{k<}^{[r]} = \prod_{\substack{ N+1\leq i\leq N+M \\ \widehat{B_{ki}} < 0}}  f_{i}^{[r]}
\end{equation*}

In what follows say $[n]:=\{1,...,n\} \subset \mathbb{N}$. Consider the effect of a mutation in the direction $k\in [n]$ on the principle part $\hbox{Principle}(B)$ of an exchange matrix $B$. Whenever we consider a particular matrix entry, say $B_{i,j}$ where $i,j \in [n]$, if $i=k$ then it's clear that $(\mu_k (B)_{i,j}$  will still be divisible by $d_i$ since the mutation only changed the sign of the entry $B_{i,j}$. On the other hand, if $i\not = k$, consider that
\begin{gather*}
    (\mu_k (B))_{i,j} = B_{i,j} + \frac{|B_{i,k}|B_{k,j}+B_{i,k}|B_{k,j}}{2}  \\= d_i \widehat{B_{i,j}} + d_i d_k\frac{|\widehat{B_{i,k}}|\widehat{B_{k,j}}+\widehat{B_{i,k}}|\widehat{B_{k,j}}|}{2} = d_i (\widehat{B_{i,j}}+d_k\frac{|\widehat{B_{i,k}}|\widehat{B_{k,j}}+\widehat{B_{i,k}}|\widehat{B_{k,j}}|}{2})
\end{gather*}
hence the mutated entry $(\mu_k (B))_{i,j}$ is divisible by $d_i$, whenever $i,j\in [n]$ so that $B_{i,j}$ is an entry of the principle part of the exchange matrix. More generally, it follows that $d_i$ will divide each entry of the $i^{th}$ row of $\hbox{Principle}(B)$ after any sequence of mutations. 

On the other hand, if we consider the effect of a mutation in direction $k\in [n]$ on the slack part $\hbox{Slack}(B)$ of the exchange matrix $B$ we get a similar formula. For instance, whenever $i\in [N]$ and $j\in [N+M]\setminus [N]$, so that $B_{i,j}$ is an entry of the slack part of the exchange matrix, then
\begin{gather*}
    (\mu_k (B))_{i,j} = B_{i,j} + \frac{|B_{i,k}|B_{k,j}+B_{i,k}|B_{k,j}|}{2} \\= \widehat{B_{i,j}} + \frac{|B_{i,k}|\widehat{B_{k,j}}+B_{i,k}|\widehat{B_{k,j}}|}{2} =  \widehat{B_{i,j}} + d_i \frac{|\widehat{B_{i,k}}|\widehat{B_{k,j}}+\widehat{B_{i,k}}|\widehat{B_{k,j}}|}{2}
\end{gather*}
 
\begin{remark}\label{remark: modified exchange matrix mutation rule}
By the preceding discussion,  it follows that the modified exchange matrix mutates by the following modified mutation rule,

\begin{equation*}
     (\mu_k (\widehat{B}))_{i,j} = \begin{cases} -\widehat{B_{i,j} }& \hbox{if }i= k \hbox{ or } j=k\\ \widehat{B_{i,j}}+ d_k \frac{|\widehat{B_{i,k}}|\widehat{B_{k,j}}+\widehat{B_{i,k}}|\widehat{B_{k,j}}|}{2} & \hbox{if } i\not = k\not =j \hbox{ and } j \leq N \\ \widehat{B_{i,j}}+ d_i \frac{|\widehat{B_{i,k}}|\widehat{B_{k,j}}+\widehat{B_{i,k}}|\widehat{B_{k,j}}|}{2} & \hbox{if } i\not =k \not = j \hbox{ and } j > N
     \end{cases}
\end{equation*}

The distinction between the second and third equations is whether $j$ is mutable or frozen; i.e, whether or not $\widehat{B_{i,j}}$ is from the principal part or the slack part of an exchange matrix. This expression for the mutation rule of the modified exchange matrix was previously used in \cite{gekhtman2022generalized}. 
\end{remark}

The mutation in direction $k$ is an involution on seeds, i.e. $(\mu_k \circ \mu_k)(\mathbf{t})=\mathbf{t}$. It follows that mutation is an equivalence relation on seeds. We say two seeds $\textbf{t}$ and $\textbf{t}'$ are mutation equivalent whenever $\mathbf{t}' = (\mu_{k_a}\circ ... \circ \mu_{k_1}) (\mathbf{t})$ for some finite set of directions $k_1,...,k_a\in \{1,...,N\}$ and we denote this equivalence relation by $\textbf{t}\sim \textbf{t}'$.

In this work, we often have multiple seeds in our discussion of exchange polynomials or stable and cluster monomials so we often denote a particular seed $\mathbf{t}=(\mathbf{x},B,\mathcal{P})$ by a superscript in our notation, for instance, $\theta^{\mathbf{t}}_{k}$ will denote the generalized exchange polynomial of the $k^{th}$ cluster variable of the seed $\mathbf{t}$. Similarly, $\mu_{k<}^{\mathbf{t}}$ and $\mu_{k>}^{\mathbf{t}}$ will denote the cluster monomials and  $v_{k<}^{\mathbf{t}}$ and $v_{k>}^{\mathbf{t}}$ will denote the stable monomials in the direction of the $k^{th}$ cluster variable. On the other hand, whenever we believe the seed $\mathbf{t}$ is entirely clear from the context of the discussion we will omit this superscript notation.

\begin{definition}
    Fix a generalized seed $\textbf{t}_0=(\text{\textbf{x}}, B, \mathcal{P})$, which we will call an \textit{initial seed} and denote the ring of Laurent polynomials in the stable variables as $\mathbb{A}=\mathbb{ZP}$. The corresponding \textit{generalized cluster algebra} $\mathcal{A}^g(t_0)$, or sometimes written as $\mathcal{A}^g(\mathbf{x}, B, \mathcal{P})$, is the $\mathbb{A}$-subalgebra of $\mathcal{F}$ generated by the cluster variables in every seed $\mathbf{t}$ that is mutation equivalent to the initial seed, that is $\mathbf{t}\sim \mathbf{t}_0$. If the set of cluster variables by is denoted by $\mathcal{X}$ then we also use the following notation for the associated GCA: $\mathcal{A}^g(\mathbf{t}_0)=\mathbb{A}[\mathcal{X}]$. \color{black}
\end{definition}

When we choose $d_i=1$ for all $i$, we recover the standard definition of a cluster algebra $\mathcal{A}(\mathfrak{t}_0)$ as defined in \cite{fomin2002cluster}. Let $D:=\prod_{=1}^N d_i$ and  $\mathcal{D}=\sum_{i=1}^N d_i$, call $D$ the \textit{Total Multiplicity} of the GCA and $\mathcal{D}$ the \textit{Pseudo-Rank} of the GCA, so that when the total multiplicity of the generalized cluster algebra is 1, we get a (traditional) cluster algebra.  

We now introduce some combinatorial language to help describe our construction in terms of special graphs.

\begin{definition}
A \textit{Quiver} is a finite directed graph $Q=(Q_0, Q_1)$ where $Q_0$ is the set of vertices and $Q_1$ the set of edges such that there are no edges from a vertex to itself and there are no pairs of edges going in opposite directions between a pair of vertices (that is, there are no 2-cycles).
\end{definition}

To any extended exchange matrix $B\in \mathbb{Z}^{N\times (N+M)}$ whose principal part $\hbox{Principal}(B)$ is a skew-symmetric matrix, there is a unique quiver $Q=(Q_0, Q_1)$. To construct this quiver, associate a vertex to each column of $B$ so that there are $N+M$ many vertices, say $Q_0=\{v_1,...,v_N,...,v_{N+M}\}$ denotes the vertex set; Whenever $B_{i,j}>0$ assign $B_{i,j}$ many arrows from $v_i$ to $v_j$ and similarly, whenever $B_{i,j}<0$ assign $|B_{i,j}|$ many arrows from $v_j$ to $v_i$. The partition of an exchange matrix into its principal and slack parts; induces a partition of the vertices into the disjoint union of the set $Q_0^{m}=\{v_1,...,v_N\}$ and the set $Q_{0}^{f}=\{v_{N+1},...,v_{N+M}\}$ so that $Q_0=Q_0^m \sqcup Q_0^f$ where $\sqcup$ denotes a disjoint union. We call the set of vertices in $Q_0^{m}$ \textit{mutatable vertices} and the set of vertices in $Q_{0}^{f}$ \textit{frozen vertices}. Likewise, for any quiver $Q=(Q_0,Q_1)$ with vertices labeled as frozen and mutatable, there exists a unique exchange matrix $B$ whose principle part is a skew-symmetric matrix, and whose entries $B_{i,j}$ has the absolute value $|B_{i,j}|$ given by the number of arrows between the vertices $v_i$ and $v_j$, and whose sign is $\hbox{sign}(B_{i,j})=1$ whenever there are arrows from $v_i$ to $v_j$ and where $\hbox{sign}(B_{i,j})=-1$ whenever there are arrows from $v_j$ to $v_i$. This provides a one-to-one correspondence between exchange matrices (whose principal parts are skew-symmetric) and quivers.\color{black}

In the general case of an (extended) exchange matrix whose principle part is skew-symmetrizable, the exchange matrix may be represented by a (weighted) quiver called a \textit{diagram}. Still, the diagram does not uniquely determine the (extended) exchange matrix, see \cite{ClusterII}.

In the case that the modified exchange matrix is skew-symmetric, we can represent it with a generalization of a quiver. 

\begin{definition}[Node Weighted Quiver]\label{def: weighted quiver}
Given a quiver $Q=(Q_0,Q_1)$ whose vertex set is partitioned into mutatable and frozen vertices so that $Q_0=Q^m \sqcup Q^f$ with $N$ mutable vertices and $M$ frozen vertices, the $(N+1)-$tuple $(Q,d_1,...,d_N)$ with $d_i\in \mathbb{N}$ is called a \textit{Node Weighted Quiver}.

Say $k\in Q^m$ is a mutable vertex, then the \textit{(Node) Weighted Quiver Mutation} in direction $k$ associated with $(Q,d_1,...,d_N)$ transforms this tuple into a new (node) weighted quiver $(Q',d_1,...,d_N)$ given by the following rules (here $i,j,k\in Q_0$
\begin{enumerate}
    \item For every path between vertices $i\to k \to j$ add $d_k$ many new arrows from $i$ to $j$ whenever both $i$ and $j$ are mutable i.e. $i,j \in Q^m$
    \item For every path between vertices $i\to k \to j$ where exactly one of $i$ or $j$ is mutatable, add $d_i$ or $d_j$ many new arrows from $i$ to $j$ depending on which vertex is mutatable. For instance, say without loss of generality $i \in Q^f$ and $j \in Q^m$ then for every path $i\to k \to j$ add $d_j$ many new arrows $i\to j$. Similarly, if $i\in Q^m$ and $j \in Q^f$ add $d_i$ new arrows $i\to j$. 
    \item Reverse the direction of all arrows indecent to vertex $k$ 
    \item Repeatedly remove all oriented $2$-cycles until all 2-cycles are removed.

\end{enumerate}
\end{definition}

Applying a node-weighted quiver mutation to a node-weighted quiver is precisely the same as applying the modified exchange matrix mutation rule from remark \ref{remark: modified exchange matrix mutation rule} to the corresponding exchange matrix associated with the quiver. Thoughout this work, whenever we discuss the node weighted quiver associated to a GCA, we are assuming the modified exchange matrix $\hat{B}$ used to generate the GCA has a skew symmetric principal part and are considering the quiver associated to this modified exchange matrix.

We now introduce some language for quivers, which will eventually help describe our constructed cluster algebra based on a particular generalized cluster algebra. It is natural to consider partitions of the vertex set of a quiver. The following definitions are from \cite{SpecialFoldings}.

\begin{definition}[Folding of a Quiver and Group Mutations]
Consider a quiver $Q=(Q_0,Q_1)$ where $Q_0$ is the vertex set, and $Q_1$ is the edge set and further consider a partition of the vertex set into some $k\in \mathbb{N}$ many equivalence classes\color{black}, say,
\begin{equation}
    Q_0 = Q_0^1 \sqcup Q_0^2 \sqcup ,...,\sqcup Q_0^k
\end{equation}
If a partition of the vertex set of a quiver satisfies the following conditions, we call this partition a \textit{Valid folding} of the quiver $Q$,
\begin{enumerate}
    \item The vertices in a given equivalence class\color{black}, such as $Q_0^i$, have no arrows amongst themselves
    \item After mutating every vertex in an equivalence class\color{black}, such as $Q_0^i$,  exactly once, property $(1)$ still holds. 
\end{enumerate}
These conditions are referred to as the \textit{folding conditions}. Often, we will call a valid folding of a quiver a folding of a quiver.

For $j \in \{ 1,...,k\}$ and a quiver $Q$ which has a valid folding as in the notation above, the \textit{group mutation} in direction $j$, associated with the equivalence class $Q_0^j$, transforms the quiver $Q$ into a new quiver $Q'$ given by mutating each vertex of $Q_0^j$, that is,
\begin{equation}
    Q' = (\circ_{j'\in Q_0^j} \mu_{j'} )(Q)
\end{equation}
Since in this definition $Q$ is assumed to have valid folding and $Q_0^j$ is an equivalence class of this folding, the composition above is independent of the order of vertices from $Q_0^j$, and this expression is well defined. 
\end{definition}

We now introduce a notation of \textit{unfolding} as defined in \cite{Unfolding}, which applies to general skew-symmetrizable matrices.

\begin{definition}[Unfolding]\label{def: unfolding}
    Let $B$ be an $n\times n$ indecomposable skew-symmetrizable matrix, and $BD$ be skew-symmetric, where $D=(d_i)$ is an integer-valued diagonal matrix with positive diagonal entries. Let $\mathcal{B}$ be a block matrix given by,
    \begin{equation*}
        \mathcal{B}= (\mathcal{B}^{i,j})_{i,j=1}^{n}
    \end{equation*}
where $\mathcal{B}^{i,j}$ is an integer-valued $d_i \times d_j$ sub-matrix so that $\mathcal{B}$ is a $\sum_{k=1}^n d_k$ by $\sum_{k=1}^n d_k$ matrix. Say $\mathcal{B}$ satisfies the following conditions with respect to $B$,
\begin{enumerate}
    \item[(1)] The sum of entries of each column of $\mathcal{B}^{i,j}$ is $B_{i,j}$.
    \item[(2)] If $B_{i,j}>0$, then the entries of the block $\mathcal{B}^{i,j}$ are all non-negative
\end{enumerate}

In the case that $\mathcal{B}$ is skew-symmetric and it satisfies the conditions above with respect to $B$, it follows that the quiver associated with $\mathcal{B}$ has a valid folding on its vertices according to the partition of the columns of $\mathcal{B}$ induced by its blocks, hence there is a set of well-defined group mutations, say $\widehat{\mu_{1}},...,\widehat{\mu_n}$. Even in the general case where $\mathcal{B}$ is only skew-symmetrizable, the mutation of $\mathcal{B}$ by the underlying columns of some block $\mathcal{B}^{i,j}$ is independent of the order of the columns. Hence, there is a well-defined set of group mutations as in the notation above, whenever $\mathcal{B}$ satisfies the conditions above.

We say $\mathcal{B}$ is an unfolding of $B$ if $\mathcal{B}$ satisfies $(1)$ and $(2)$ as above and for any sequence of mutations $\mu_{k_1} \circ ... \circ \mu_{k_m}$, the matrix $\mathcal{B}' = \widehat{\mu_{k_1}} \circ ... \circ \widehat{\mu_{k_m}}(\mathcal{B})$ satisfies $(1)$ and $(2)$ with respect to $B' = \mu_{k_1} \circ ... \circ \mu_{k_m}(B)$.
\end{definition}

\begin{remark}
    Let $\widehat{\mu_k} = \mu_{k_1} ... \mu_{k_m}$ be the  group mutation for the group $K$ of $\mathcal{B}= (\mathcal{B}^{i,j})_{i,j=1}^{n}$. Using properties (1) and (2) of definition \ref{def: unfolding}, for any two groups  $Y=\{y_1,...,y_{d_y}\} $ and $Z=\{z_1,\cdots,z_{dz}\}$, we compute $\widehat{\mu_k}$ for each entry of the block $\mathcal{B}^{Y,Z} = \{b_{y_i,z_j}\}$: 
    \begin{align*}
        \widehat{\mu_k}(b_{y_i,z_j}) = \begin{cases}
            -b_{y_i,z_j} & \text{if $K = Y$ or $K = Z$} \\
            b_{y_i,z_j} + \frac{1}{2}\sum_\ell\left(\text{sgn}(b_{y_i,k_\ell})b_{y_i,k_\ell} b_{k_\ell, z_j} + \text{sgn}(b_{k_\ell, z_j})b_{y_i,k_\ell} b_{k_\ell, z_j}\right) & \text{otherwise} 
        \end{cases}
    \end{align*}
    Now property (2) tells us that the signs of each block are well-defined and thus
    \begin{align*}
        \widehat{\mu_k}(b_{y_i,z_j}) = \begin{cases}
            -b_{y_i,z_j} & \text{if $K = Y$ or $K = Z$} \\
            b_{y_i,z_j} + \frac{1}{2}\sum_\ell\left(\text{sgn}\left(\mathcal{B}^{Y,K}\right)b_{y_i,k_\ell} b_{k_\ell, z_j} + \text{sgn}\left(\mathcal{B}^{K,Z}\right)b_{y_i,k_\ell} b_{k_\ell, z_j}\right) & \text{otherwise} 
        \end{cases}        
    \end{align*}
    Hence we can compute the mutation $\widehat{\mu_k}$ on each block $(\mathcal{B}^{Y,Z})$:
    \begin{align} \label{eq : grouping_mutation}
        \widehat{\mu_k}\left(\mathcal{B}^{Y,Z}\right) = \begin{cases} - \mathcal{B}^{Y,Z} & \text{$K = Y$ or $K=Z$} \\ \mathcal{B}^{Y,Z} + \frac{1}{2}\left(\text{sgn}(\mathcal{B}^{Y,K})\mathcal{B}^{Y,K}\mathcal{B}^{K,Z} + \text{sgn}\left(\mathcal{B}^{K,Z}\right)\mathcal{B}^{Y,K}\mathcal{B}^{K,Z}\right) & \hbox{otherwise.} \end{cases}
    \end{align}
    
    In other words, we can obtain the mutation formula for the group mutation $\widehat{\mu_k}$ by replacing the entries $B_{i,j}$ with the corresponding blocks $\mathcal{B}^{i,j}$ in the standard mutation rule. See definition \ref{def:seed_mutation}, part (2) for comparison.
\end{remark}

A fundamental result of cluster algebras is that any cluster may be written as a Laurent polynomial in the cluster variables of the initial seed, this is called the \textit{Laurent Property}. The proof for generalized cluster variables is discussed in \cite{Techmuller}.

\begin{theorem}[Laurent Phenomenon]
    Any generalized cluster variable may be written as a Laurent polynomial in the cluster variables of the initial seed.
\end{theorem}

Say $\mathcal{A}^g(\mathbf{t}_0)$ is a generalized cluster algebra over ground ring $\mathbb{ZP}$. In what follows, let $\mathcal {A}_{\mathbf{t}}^{g}$ denote the $\mathbb{ZP}$ sub algebra of the ambient field $\mathcal{F}$ generated by the cluster variables of $\mathbf{t}$. Note that for any seed $\mathbf{t}\sim \mathbf{t}_0$, its stable monomials and cluster monomials are in this ring, that is for any $k\in [N]$ then $(\mu_{k<})^{\mathbf{t}},(\nu_{k<}^{[r]})^{\mathbf{t}} \in \mathcal{A}_{\mathbf{t}}^g$ and it follows that the exchange polynomials of the seed are also in this ring, hence $(\theta_k)^{\mathbf{t}}\in \mathcal{A}_{\mathbf{t}}^g$.

In \cite{DrinfeldDouble} a rewritten expression for a generalized exchange polynomial was considered, which we now derive. As in the notation of \cite{DrinfeldDouble}, in what follows write $v_{k>}$ instead of $v_{k>}^{[d_k]}$ and $v_{k<}$ instead of $v_{k<}^{[d_k]}$. For integers $r\in \{0,...,d_i\}$ and $k\in \{1,...,N\}$ consider the following Laurent monomials in the stable variables,
\begin{equation*}
    q_{kr} = \frac{(v_{k>})^{r}(v_{k<})^{d_k-r}}{(v^{[r]}_{k>} v^{[d_k-r]}_{k<})^{d_k}} \in \mathbb{P}
\end{equation*}
It was shown in \cite{DrinfeldDouble} that $q_{kr}$ mutates the same as the rule \ref{eq: p mutation rule}. In what follows we use the Laurent monomial $\widehat{p_{kr}}:=\frac{(p_{kr}^{d_k})}{q_{kr}}\in \mathbb{P}$. Now the generalized exchange relation from \ref{eq: Generalized Exchange Relation} can be rewritten as
\begin{equation}\label{eq: Root Formula}
    \theta_k  = \sum_{r=0}^{d_k} (\widehat{p_{kr}} (v_{k>})^{r} (v_{k<})^{d_k-r})^{\frac{1}{d_k}} (u_{k>})^{r} (u_{k<})^{d_k-r}
\end{equation}
since 
\begin{equation*}
    \widehat{p_{kr}} (v_{k>})^{r} (v_{k<})^{d_k-r} = (p_{kr}v^{[r]}_{k;>} v^{[d_k-r]}_{k;<})^{d_k}\in \mathbb{P}
\end{equation*}
We refer to \ref{eq: Root Formula} as the \textit{Root formula}. We now define laurent monomials in the semifield $\mathbb{P}$, which will sometimes be related to $q_{kr}$. 

In what follows, let $f_j$  be a frozen variable (here $N+1\leq j \leq N+M$) of the (generalized) cluster algebra $\mathcal{A}(\mathbf{t}_0)$ with (initial) modified exchange matrix $\widehat{B}$. Given any (generalized) seed $\mathbf{t}$ which is mutation equivalent to $\mathbf{t}_0$ and cluster variable $x_k^{\mathbf{t}}$ of this seed, for any natural numbers $n\in \mathbb{N}$ and $r\in [d_k]$ we define the following Laurent monomial in $\mathbb{P}$ associated to the frozen variable $f_j$ 
\begin{equation} \label{def: special laurent monomials}
    {}_{\mathbf{t}}^{n}(f_j)_{k}^{r}=  f_j^{\left( n \left\lfloor \frac{r \cdot \widehat{B_{kj}}}{d_k} \right\rfloor - \left\lfloor n \cdot \frac{r \cdot \widehat{B_{kj}}}{d_k} \right\rfloor \right)} \in \mathbb{P}
\end{equation}

Note that the particular Laurent monomials ${}_{\mathbf{t}}^{d_k}(f_j)_{k}^{r}$ have the following relations for any natural number $e\in \mathbb{N}$,
\begin{equation}\label{eq: section one relations}
    \left[{}_{\mathbf{t}}^{d_k} (f_j)_{k}^{r} \right]^e =     {}_{\mathbf{t}}^{e\cdot d_k} (f_j)_{k}^{r} 
\end{equation}

In addition, each ${}_{\mathbf{t}}^{d_k}(f_j)_{k}^{r}$ will be related to $q_{kr}$ in the following sense,
\begin{equation}\label{eq: q monomial generalization}
q_{kr}^{\mathbf{t}}=\prod_{N+1\leq j \leq N+M} \left[ \frac{1}{{}_{\mathbf{t}}^{d_k}(f_j)_{k}^{r}} \right]
\end{equation}

\section{Main Result}

Given a generalized seed $\mathbf{t}_0=(\mathbf{x},B,\mathcal{P})$ and an associated generalized cluster algebra $\mathcal{A}^g(\mathbf{t}_0)$ over a coefficient ring $\mathbb{A}$, we will construct a new seed $\mathbf{s}_0=(\mathbf{y},\mathcal{B}, \mathcal{P}')$ so that the associated generalized cluster algebra over a coefficient ring $\mathbb{A}'$, denoted by $\mathcal{A}^c(\mathbf{s}_0)$, has total multiplicity $1$, hence it will be a (traditional) cluster algebra. We will then construct a quotient algebra of $\mathcal{A}^c(\mathbf{s}_0)$ so that this quotient algebra contains $\mathcal{A}^g(\mathbf{t}_0)$ as an embedded $\mathbb{A}$-subalgebra and so that the embedding map $\Psi$ is compatible with mutations in the sense that a single generalized cluster mutation corresponds to a sequence of (traditional) cluster mutations. This will imply the main theorem.

In general, the generalized exchange polynomials of the given GCA  $\mathcal{A}^g(B,\mathcal{P})$ are homogeneous polynomials, although the corresponding inhomogeneous polynomial has coefficients given by certain floor functions. We show in the first section, that we can embed $\mathcal{A}^g(B,\mathcal{P})$ into a new generalized cluster algebra $\mathcal{A}^g(\widetilde{B},\widetilde{\mathcal{P}})$ such that the generalized exchange polynomials of $\mathcal{A}^g(\widetilde{B},\widetilde{\mathcal{P}})$ can be considered as homogeneous polynomials whose corresponding inhomogeneous polynomial has coefficients which are given by expressions which do not involve these floor functions previously described. The elements of the strings in $\widetilde{\mathcal{P}}$ will be called \textit{generalized coefficients}. This embedding will be an intermediate step in our overall embedding of $\mathcal{A}^g(B,\mathcal{P})$ into the mentioned quotient algebra.

Next, we consider the initial exchange matrix $B$ of a given (GCA) $\mathcal{A}^g(B,\mathcal{P})$ and describe the construction of a new exchange matrix $\mathcal{B}$ with a corresponding grouping of columns and rows. We prove some mutation compatibility conditions between mutations of the given exchange matrix $B$ and group mutations of the constructed exchange matrix $\mathcal{B}$. When we consider the associated (traditional) cluster algebra $\mathcal{A}^c(\mathbf{s}_0)$, its ring of coefficients $\mathbb{A}'$ will then contain the given ring of coefficients $\mathbb{A}$ as an embedded subring. We then define an ideal $\mathcal{I}$ of, which we will describe later, of $\mathcal{A}^c(\mathbf{s}_0)$ generated by some algebraic expression involving the generalized coefficients of the ring $\mathbb{A}$.

We will show that the quotient algebra $\mathcal{A}^c(\mathbf{s}_0) \diagup \mathcal{I}$ will contain the given generalized cluster algebra $\mathcal{A}^g(\mathbf{t}_0)$ as an embedded $\mathbb{A}$ subalgebra. To do this, we will consider the seeds of $\mathcal{A}^c(\mathbf{s}_0)$ obtained from the initial seed by group mutations. For these seeds that are group mutation equivalent to the initial seed, we prove a product formula, which we will describe later, for the product of exchange polynomials of a group of cluster variables in this quotient algebra. We show that the terms of this product have particular constant factors, which are ultimately given by the \textit{generalized coefficients}.

Based on the Laurent phenomena, we construct an embedding of the intermediate generalized cluster algebra $\mathcal{A}^g(\widetilde{B},\widetilde{\mathcal{P}})$ into the quotient algebra $\mathcal{A}^c(\mathcal{B})\diagup\mathcal{I}$. This will imply that the given generalized cluster algebra $\mathcal{A}^g(B,\mathcal{P})$ will be isomorphic to a $\mathbb{ZP}$ algebra subquotient of the cluster algebra $\mathcal{A}^c(\mathcal{B})$, provided we restrict the scalars of the (traditional) cluster algebra. We will explicitly describe the subquotient which the given generalized cluster algebra is isomorphic to as $\mathbb{ZP}$ algebras and also describe how we restrict the scalars of the (traditional) cluster algebra.

\subsection{Adjoining nth roots of frozen variables to a GCA $\mathcal{A}(\mathbf{x},B,\mathcal{P})$} \label{section: Ring of Coefficients}

\begin{figure}[b]
 \begin{subfigure}{0.49\textwidth}
     \[\begin{tikzcd}
	{\color{blue}a} &&& {x_3} \\
	\\
	\\
	{\color{blue}b} &&& {y_2}
	\arrow["4", from=1-1, to=1-4]
	\arrow["2"', from=1-4, to=4-1]
	\arrow[from=1-4, to=4-4]
	\arrow["3", from=4-1, to=4-4]
\end{tikzcd}\]
     \caption{Initial node weighted quiver $Q$}
     \label{fig:a}
 \end{subfigure}
 \hfill
 \begin{subfigure}{0.49\textwidth}
     \[\begin{tikzcd}
	{\color{blue}a^{\frac{1}{6}}} &&& {\overline{x}_3} \\
	\\
	\\
	{\color{blue}b^{\frac{1}{6}}} &&& {\overline{y}_2}
	\arrow["{\color{red}24}", from=1-1, to=1-4]
	\arrow["{\color{red}12}"', from=1-4, to=4-1]
	\arrow[from=1-4, to=4-4]
	\arrow["{\color{red}18}", from=4-1, to=4-4]
\end{tikzcd}\]
     \caption{Constructed node weighted quiver $\overline{Q} $ }
     \label{fig:b}
 \end{subfigure}

 \begin{subfigure}{0.49\textwidth}
     $\theta_x= a^4+p_{1x}a^2y+p_{2x}ay^2 b+y^3b^2$ \\
     $\theta_y = b^3x^2+p_{1y}bx+1$
     \caption{Generalized Exchange Polynomials with floor functions in coefficients}
     \label{fig:a}
 \end{subfigure}
 \hfill
 \begin{subfigure}{0.49\textwidth}
     $\theta_x= a^{24}+\frac{(p_{1x})^6}{a^4}a^{16} yb^4+\frac{(p_{2x})^6}{a^2}a^8 y^2 b^8 +y^3 b^{12}$ \\
     $\theta_y = b^{18}x^2+\frac{(p_{1y})^6}{b^3}b^9x+1$
     \caption{New Generalized Exchange Polynomials without floor functions in coefficients}
     \label{fig:b}

 \end{subfigure}

\begin{subfigure}{0.49\textwidth}
     $\tau_x =\frac{y}{1} \hbox{ and } \tau_y = \frac{1}{x}$
     \caption{Homogeneous variables of given exchange polynomials}
     \label{fig:a}
 \end{subfigure}
 \hfill
 \begin{subfigure}{0.49\textwidth}
     $\tau_x =\frac{b^4y}{a^8} \hbox{ and } \tau_y = \frac{1}{xb^9}$
     \caption{Homogeneous variables of new exchange polynomials}
     \label{fig:b}
 \end{subfigure}
    \caption{A given node weighted quiver $Q$ of a modified exchange matrix $\hat{B}$ and the constructed the node weighted quiver $\overline{Q}$ whose exchange polynomials are homogenous without floor function as coefficient.} \label{fig: quiver with adjoined roots}
\end{figure}
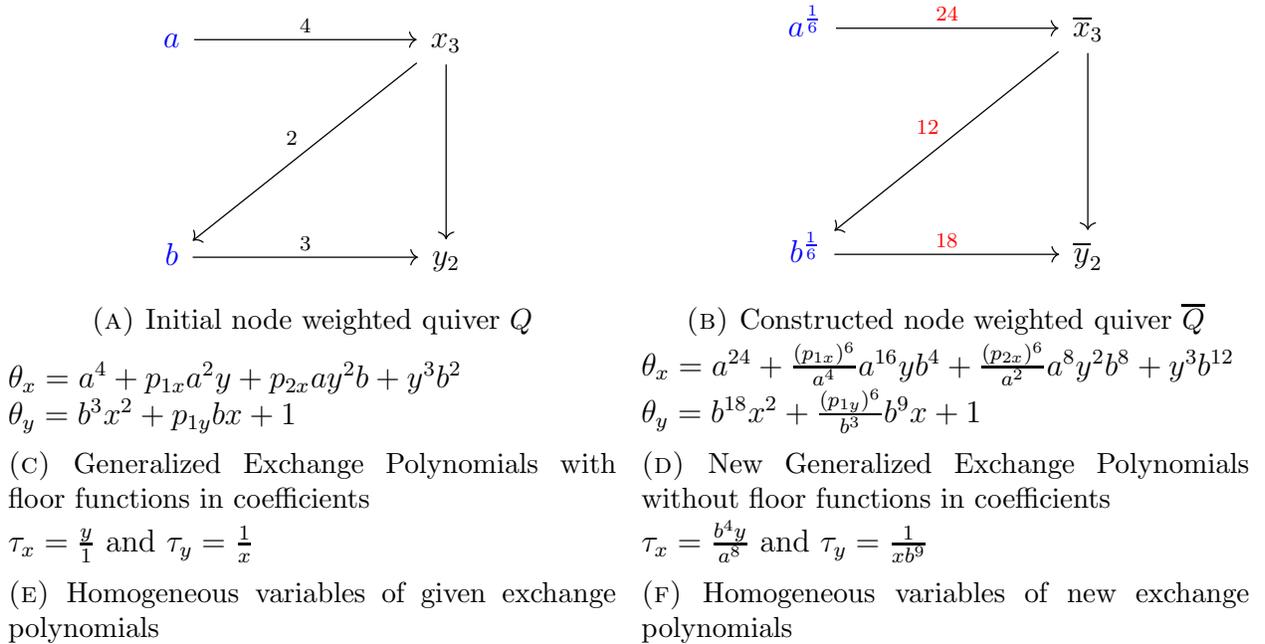

One of the biggest difficulties in working with the generalized exchange relation (\ref{eq: Generalized Exchange Relation}) is using the integer floor function to compute the exponents of the frozen variables. In this section, we will show that for any generalized cluster algebra $\mathcal{A}(\mathbf{x}, B, \mathcal{P})$ and frozen variable $f_j$, the algebra $\mathcal{A}(\mathbf{x}, B, \mathcal{P})[f_j^{1/n}]$ obtained by replacing $f_j$ with its $n$-th root is identifiable as a generalized cluster algebra. By composing these embeddings for each frozen variable, with a decent choice of $n$, we can consider any generalized cluster algebra in a larger generalized cluster algebra such that the exchange polynomials can be considered as homogenous polynomials whose coefficients do not contain floor functions. The strings of coefficients of this generalized cluster algebra will be used to define the subquotient in Section $\ref{sec: embedding}$. See figure $\ref{fig: quiver with adjoined roots}$ for an example of this process of adjoining roots to remove integer floor functions from the coefficients. 

In what follows, say $[n]=\{1,..,n\}$.

\begin{definition}
    Let $\mathcal{A}(\mathbf{x}, B, \mathcal{P})$ be a generalized cluster algebra over the ground ring $\mathbb{ZP}$ and semifeild $\mathbb{P}=\hbox{trop}(f_1,...,f_m)$. Let $\varphi_{j,n}:\mathbb{ZP}\to \mathbb{ZP}$ denote the unique embedding of $\mathbb{Z}$ algebras which fixes every frozen variable except for $f_j$, which is sent to $(f_i)^n$; that is;
    \begin{equation*}
        \varphi_{j,n}(f_i) = \begin{cases}
            (f_i)^n & \quad i=j, \\
            f_i & \quad i \neq j.
        \end{cases} 
    \end{equation*}
\end{definition}

In what follows, recall the special Laurent monomials ${}_{\mathbf{t}_0}^n(f_j)_{k}^{r}$ associated to the frozen variables of a (generalized) cluster algebra $\mathcal{A}(\mathbf{t}_0)$ defined in equation \ref{def: special laurent monomials} of the preliminaries.

\begin{lemma}
    Let $\mathbf{t}_0=(\mathbf{x},B,\mathcal{P})$ of rank $N$ and for a sequence of direction $s\in  [N]^d$ of length $d$ and consider the generalized seed $\mathbf{t}=\mu_s(\mathbf{t}_0)$. There exists a generalized seed $\overline{\mathbf{t}_0}=(\overline{\mathbf{t}_0},\overline{B},\overline{\mathcal{P}})$ such that the map $\varphi_{j,n}$ uniquely extends to an embedding of $\mathbb{ZP}$ algebras $\varphi_{j,n} : \mathcal{A}(\mathbf{t}_0) \hookrightarrow \mathcal{A}(\overline{t_0})$ where for each frozen variable $f_j$ and cluster variable $x_k^{\mathbf{t}}$, the map is given by 
    \begin{align*}
        \varphi_{j,n}(f_i) = 
    \begin{cases}  (f_i)^n  &i = j, \\ f_i & i \neq j. \end{cases}\\
        \varphi_{j,n}(x_k^{\mathbf{t}})=\overline{x_k}^{\overline{\mathbf{t}}} 
    \end{align*}
    where $\overline{\mathbf{t}}=\mu_s(\overline{\mathbf{t}_0})$. Moreover, $\overline{\mathbf{t_0}}$ is given by the following data
    \begin{enumerate}
        \item $\overline{B}$ is the matrix obtained from $B$ by multiplying the column $j$ corresponding to $f_j$ by $n$,
        \item The coefficients $\overline{p_{kr}} \in\overline{P}$ are given by the formula
        \begin{align*}
            \overline{p_{kr}} := \varphi_{j,n}(p_{kr}) \cdot {}_{\mathbf{t}_0}^n(f_j)_{k}^{r}.
        \end{align*}
    \end{enumerate}
\end{lemma}

\begin{proof}

    There is a unique embedding of $\mathbb{Z}$ algebras $\varphi:\mathcal{F}\to \overline{\mathcal{F}}$ such that $\varphi(f_i)=\varphi_{j,n}(f_i)$ and $\varphi(x_k^{\mathbf{t}_0})=\overline{x_k}^{\overline{\mathbf{t}_0}}$, on the hand $\varphi $ is also a morphism of $\mathbb{ZP}$ algebras when its domain and codomain have the $\mathbb{ZP}$ algebra structure given by the following diagram:
    \[\begin{tikzcd}
	{\mathbb{ZP}} && {\mathcal{F}}&& {\overline{\mathcal{F}}}
	\arrow["\iota", from=1-1, to=1-3]
	\arrow["\varphi", from=1-3, to=1-5]
    \end{tikzcd}\]

    Since $\hbox{image}(\phi\circ \iota)\subset \mathcal{A}(\overline{\mathbf{t}_0}) $ we can consider $\mathcal{A}(\overline{\mathbf{t}_0})$ as a $\mathbb{ZP}$ sub algebra of $\overline{\mathcal{F}}$ with respect to its $\mathbb{ZP}$ algebra structure given by the diagram above. We now claim that,
    \begin{equation}\label{eq: section one generator}
        \varphi(\mathcal{X})\subset \mathcal{A}(\overline{\mathbf{t}_0})
    \end{equation}

    Say $f\in \mathcal{A}(\mathbf{t}_0)$, if this claim is true then $\varphi$ restricts to an embedding of $\mathbb{ZP}$ algebras $\varphi_{n,j}:\mathcal{A}(\mathbf{t}_0)\to \mathcal{A}(\overline{\mathbf{t}_0})$ given by $\varphi_{n,j}(f)=\varphi(f)$ since $\mathcal{A}(\mathbf{t}_0)$ is generated by $\mathcal{X}$ as a sub algebra of $\mathcal{F}$. To prove the claim \ref{eq: section one generator}, we will show that for all seeds $\mathbf{t}$, and all indices $k$ corresponding to mutable variables, and all $r\in \{1,...,d_k\}$

    \begin{enumerate}
        \item[(i)] $\varphi\left(\left(u_{k>}^{r}\right)^{(\mathbf{t})} \left(u_{k<}^{d_k-r}\right)^{(\mathbf{t})} \right) = \left(\overline{u}_{k>}^{r}\right)^{(\overline{\mathbf{t}})} \left(\overline{u}_{k<}^{d_k-r}\right)^{(\overline{\mathbf{t}})}$,
        \item[(ii)] $\varphi\left(p_{kr}^{(\mathbf{t})}\left(v_{k>}^{[r]}\right)^{(\mathbf{t})} \left(v_{k<}^{[d_k-r]}\right)^{(\mathbf{t})} \right) = \overline{p_{kr}}^{(\mathbf{\overline{t}})}\left(\overline{v}_{k>}^{[r]}\right)^{(\overline{\mathbf{t}})} \left(\overline{v}_{k<}^{[d_k-r]}\right)^{(\overline{\mathbf{t}})}$.

        \item[(iii)]
        $\varphi(X_k^{(\mathbf{t})})=\overline{X_k}^{(\overline{\mathbf{t}})}$
    \end{enumerate}

    We will assume these three equations hold for an arbitrary seed $\mathbf{t}$ and show that these three equations hold for any adjacent seed $\mathbf{t}'$, say $\mu_l(\mathbf{t})=\mathbf{t}'$. These are true for our initial seeds by construction. Conditions $(i)$ and $(ii)$ imply that $\varphi\left(\theta_k^{(\mathbf{t})}\right) = \overline{\theta}_k^{(\overline{\mathbf{t}})}$ and even more that $\varphi\left(X_k^{(\mathbf{t'})}\right) = \overline{X}_k^{(\overline{\mathbf{t'}})}$. It suffices to show conditions $(i)$ and $(ii)$ hold for $\mathbf{t}'$.
    
    The mutation rule on $B$ commutes with the operation of multiplying a (frozen) column by $n$ to obtain $\overline{B}$. Hence the exponents for the corresponding cluster variables are the same and therefore equation $(i)$ holds for $\mathbf{t}'$ by our assumption that equation $(iii)$ holds for $\mathbf{t}$.

    If $k=\ell$ then by the mutation rule $p_{k,r}^{(\mathbf{t}')}\left(v_{k>}^{[r]}\right)^{(\mathbf{t}')} \left(v_{k<}^{[d_k-r]}\right)^{(\mathbf{t}')}=p_{k,d_k-r}^{(\mathbf{t})}\left(v_{k<}^{[r]}\right)^{(\mathbf{t})} \left(v_{k>}^{[d_k-r]}\right)^{\mathbf{t}}$ and similarly $\overline{p_{k,r}}^{(\mathbf{t}')}\left(\overline{v}_{k>}^{[r]}\right)^{(\mathbf{t}')} \left(\overline{v}_{k<}^{[d_k-r]}\right)^{(\mathbf{t}')}=\overline{p_{k,d_k-r}}^{(\mathbf{t})}\left(v_{k<}^{[r]}\right)^{(\mathbf{t})} \left(\overline{v}_{k>}^{[d_k-r]}\right)^{\mathbf{t}}$, so condition $(ii)$ follows for seed $\mathbf{t}'$ by our assumption that condition $(ii)$ holds for seed $\mathbf{t}$.

    Now suppose that $k\not =\ell$. By the mutation rule $p_{k,r}^{(\mathbf{t}')}=p_{k,r}^{(\mathbf{t})}$ and $\overline{p_{k,r}}^{(\mathbf{t}')}=\overline{p_{k,r}}^{(\mathbf{t})}$ so it follows that $\frac{\overline{p_{k,r}}^{(\overline{\mathbf{t}}')}}{\varphi_{j,n}\left(p_{k,r}^{(\mathbf{t}')}\right)}=\frac{\overline{p_{k,r}}^{(\overline{\mathbf{t}})}}{\varphi_{j,n}\left(p_{k,r}^{(\mathbf{t})}\right)}$. Without loss of generality assume that $B_{k,N+j}^{(\mathbf{t})}>0$ and $B_{k,N+j}^{(\mathbf{t}_0)}>0$ so that if the exchange matrix has a corresponding quiver then there exists an edge from the mutatable vertex $k$ to the frozen $j$ vertex at both the given seed $\mathbf{t}$ and the initial seed. Now consider that
    \begin{equation*}
    \frac{\varphi\left(\left(v_{k>}^{[r]}\right)^{(\mathbf{t})} \left(v_{k<}^{[d_k-r]}\right)^{(\mathbf{t})} \right)}{\left(\overline{u}_{k>}^{r}\right)^{(\overline{\mathbf{t}})} \left(\overline{u}_{k<}^{d_k-r}\right)^{(\overline{\mathbf{t}})}} =\frac{f_j^{ n \left\lfloor \frac{r \cdot B_{k,N+j}^{(\mathbf{t})}}{d_k} \right\rfloor  }}{f_j^{\left\lfloor n \frac{r \cdot B_{k,N+j}^{(\mathbf{t})}}{d_k}\right\rfloor }}
    \end{equation*}
    We now consider two cases based on the sign of $B_{k,N+j}^{(\mathbf{t'})}$. There exists integers $c_1,c_2\in \mathbb{Z}$ such that $B_{k,N+j}^{(\mathbf{t})}=B_{k,N+j}^{(\mathbf{t}_0)}+c_1 d_k$ and $B_{k,N+j}^{(\mathbf{t'})}=B_{k,N+j}^{(\mathbf{t}_0)}+c_2 d_k$. If $B_{k,N+j}^{(\mathbf{t'})}>0$ then it follows that 
    \begin{equation*}
         \frac{\varphi\left(\left(v_{k>}^{[r]}\right)^{(\mathbf{t}')} \left(v_{k<}^{[d_k-r]}\right)^{(\mathbf{t}')} \right)}{\left(\overline{u}_{k>}^{r}\right)^{(\overline{\mathbf{t}}')} \left(\overline{u}_{k<}^{d_k-r}\right)^{(\overline{\mathbf{t}}')}}=\frac{f_j^{ n \left\lfloor \frac{r \cdot B_{k,N+j}^{(\mathbf{t}')}}{d_k} \right\rfloor  }}{f_j^{\left\lfloor n \frac{r \cdot B_{k,N+j}^{(\mathbf{t}')}}{d_k}\right\rfloor }} = \frac{f_j^{ n \left\lfloor \frac{r \cdot \left(B_{k,N+j}^{\mathbf{t}_0}+c_2 d_k\right)}{d_k} \right\rfloor  }}{f_j^{\left\lfloor n \frac{r \cdot \left(B_{k,N+j}^{\mathbf{t}_0}+c_2 d_k\right)}{d_k}\right\rfloor }}= \frac{f_j^{n\cdot r \cdot c_2}f_j^{ n \left\lfloor \frac{r \cdot \left(B_{k,N+j}^{\mathbf{t}_0}\right)}{d_k} \right\rfloor  }}{f_j^{n\cdot r \cdot c_2}f_j^{\left\lfloor n \frac{r \cdot \left(B_{k,N+j}^{\mathbf{t}_0}\right)}{d_k}\right\rfloor }}
    \end{equation*}
    and by similar reasoning it follows that 
    \begin{equation*}
        \frac{f_j^{ n \left\lfloor \frac{r \cdot B_{k,N+j}^{(\mathbf{t})}}{d_k} \right\rfloor  }}{f_j^{\left\lfloor n \frac{r \cdot B_{k,N+j}^{(\mathbf{t})}}{d_k}\right\rfloor }}=\frac{f_j^{ n \left\lfloor \frac{r \cdot \left(B_{k,N+j}^{(\mathbf{t}_0)}\right)}{d_k} \right\rfloor  }}{f_j^{\left\lfloor n \frac{r \cdot \left(B_{k,N+j}^{(\mathbf{t}_0)}\right)}{d_k}\right\rfloor }}
    \end{equation*}
    This implies that condition $(ii)$ holds in this case. If $B_{k,N+j}^{(\mathbf{t'})}<0$ then 
    \begin{equation*}
        \frac{\varphi\left(\left(v_{k>}^{[r]}\right)^{(\mathbf{t}')} \left(v_{k<}^{[d_k-r]}\right)^{(\mathbf{t}')} \right)}{\left(\overline{u}_{k>}^{r}\right)^{(\overline{\mathbf{t}}')} \left(\overline{u}_{k<}^{d_k-r}\right)^{(\overline{\mathbf{t}}')}} =\frac{f_j^{ n \left\lfloor \frac{\left(d_k-r\right) \cdot \left( c_2d_k-B_{k,N+j}^{(\mathbf{t}_0)}\right)}{d_k} \right\rfloor  }}{f_j^{\left\lfloor n \frac{\left(d_k-r\right) \cdot \left( c_2d_k-B_{k,N+j}^{(\mathbf{t}_0)}\right)}{d_k}\right\rfloor }}=\frac{f_j^{n\left(c_2d_k-rc_2+B_{k,N+j}^{(\mathbf{t}_0)}\right)}f_j^{ n \left\lfloor \frac{r \cdot B_{k,N+j}^{(\mathbf{t}_0)})}{d_k} \right\rfloor  }}{f_j^{n\left(c_2d_k-rc_2+B_{k,N+j}^{(\mathbf{t}_0)}\right)}f_j^{\left\lfloor \frac{r \cdot B_{k,N+j}^{(\mathbf{t}_0)})}{d_k} \right\rfloor }}
    \end{equation*}
    which implies that condition $(ii)$ holds for seed $\mathbf{t}'$ regardless of the sign of $B_{k,N+j}^{(\mathbf{t'})}$.

\end{proof}

In the following remark we make the identification of the algebra $\mathcal{A}(\mathbf{t}_0)[f_j^{\frac{1}{n}}]$ with a generalized cluster algebra precise. 

\begin{remark}
    If $\mathcal{X}$ denotes the cluster variables of $\mathcal{A}(\mathbf{x},B,\mathcal{P})$ so that $\mathcal{A}(\mathbf{x},B,\mathcal{P})=\mathbb{ZP}[\mathcal{X}]$ then $\mathcal{A}(\mathbf{x},B,\mathcal{P})[f_j^{\frac{1}{n}}]=\mathbb{ZP}[f_j^{\frac{1}{n}}][\mathcal{X}]$. The structure map of the $\mathbb{ZP}$ algebra $\mathcal{A}(\mathbf{x},B,\mathcal{P})[f_j^{\frac{1}{n}}]$ is then given by the inclusion map $\mathbb{ZP}\to\mathbb{ZP}[f_j^{\frac{1}{n}}]=\mathbb{ZP}[F]\diagup<F^n-f_j>\to \mathcal{A}(\mathbf{x},B,\mathcal{P})[f_j^{\frac{1}{n}}] $

    There is a unique isomorphism of $\mathbb{Z}$-algebras 
    \begin{equation*}
        \Phi:\mathcal{A}(\mathbf{x}, \overline{B}, \overline{\mathcal{P}}) \to \mathcal{A}(\mathbf{x}, B, \mathcal{P})\left[f_j^{\frac{1}{n}}\right]
    \end{equation*}
    such that $\Phi (f_j) =(f_j)^{\frac{1}{n}}$, $\Phi(f_i)=f_i$ whenever $i\not =j$, and $\Phi(\overline{x_k}^{\overline{\mathbf{t}}})=x_k^{\mathbf{t}}$. The map $\Phi$ is an isomorphism of $\mathbb{ZP}$ algebras when its domain and codomain have the previously mentioned $\mathbb{ZP}$ algebra structures. We therefore identify $\mathcal{A}(\mathbf{x}, B, \mathcal{P})\left[f_j^{\frac{1}{n}}\right]$ as the cluster algebra $\mathcal{A}(\mathbf{x}, \overline{B}, \overline{\mathcal{P}})$ with restricted scalars via $\varphi_{j,n}:\mathbb{ZP}\to \mathbb{ZP}\subset \mathcal{A}(\mathbf{x}, \overline{B}, \overline{\mathcal{P}})$
\end{remark}

In the following remark we make our statement that for a "decent" choice of n, adjoining the nth root of a generalized cluster algebra creates a generalized cluster algebra without floor functions in the coefficients of the exchange relations. For each frozen variable $f_j$ and $r\in \mathbb{N}$ recall the Laurent monomial $f_j^{[r]}$ defined in \ref{eq: frozen box} associated to a cluster variable $x_k^{\mathbb{t}}$.

\begin{remark}\label{remark: remove floor functions}
    Let $\mathbf{t}\sim \mathbf{t}_0$ and recall that the coefficients of the inhomogenous polynomial corresponding to $\theta_k^{\mathbf{t}}$ in the variable $\tau_k^{\mathbf{t}}=\left[\frac{ u_{k>}}{u_{k<}} \right]^{\mathbf{t}}$ are $\left[p_{kr}v_{k>}^{[r]}v^{[d_k-r]}_{k<}\right]^{\mathbf{t}}$ where the factors $\left[v_{k>}^{[r]}v^{[d_k-r]}_{k<}\right]^{\mathbf{t}}$ contain floor functions.

    Without loss of generality, say $f_j$ divides $v_{k>}^{\mathbf{t}}$ for some cluster variable $x_k^{\mathbf{t}}$ of $\mathcal{A}(\mathbf{t}_0)$. Whenever the degree $d_k$ of a cluster variable $x_k^{\mathbf{t}}$ divides $n$ then the exchange polynomial $\theta_k^{\bar{\mathbf{t}}}$ will be homogeneous in the variables $\left[f_j^{[1]} u_{k>}\right]^{\overline{\mathbf{t}}}$ and $\left[u_{k<}\right]^{\overline{\mathbf{t}}}$. The corresponding inhomogeneous polynomial in the variable $\tau_k^{\overline{\mathbf{t}}}= \left[ \frac{ f_j^{[1]} u_{k>}}{u_{k<}} \right]^{\overline{\mathbf{t}}}$ will have the coefficients $\left[ \overline{p_{kr}}\frac{v_k^{[r]}v_k^{[d_k-r]}}{\left(v_k^{[r]}v_k^{[d_k-r]} \right)\rvert_{\widehat{f_j}=1}} \right]^{\overline{\mathbf{t}}}$ where $\left(v_k^{[r]}v_k^{[d_k-r]} \right)\rvert_{\widehat{f_j}=1}\in \mathbb{P}$ denotes the Laurent monomial in only the frozen variable $f_j$ obtained by evaluating all frozen variables of $\left[ v_{k>}^{[r]} v_{k<}^{[d_k-r]} \right]^{\overline{\mathbf{t}}}$ at $1$ except $f_j$, hence:
    \begin{equation*}
    \theta_k^{\bar{\mathbf{t}}}= \left[(u_k)^{d_k}(\sum_{r=0}^{d_k} \overline{p_{kr}}\frac{v_k^{[r]}v_k^{[d_k-r]}}{\left(v_k^{[r]}v_k^{[d_k-r]} \right)\rvert_{\widehat{f_j}=1}} \tau_k) \right]^{\overline{t}}
    \end{equation*}

    The coefficients of the inhomogenous polynomial of $\theta_k^{\overline{t}}$ no longer have floor functions which contain the frozen variable $f_j$. It is this sense in which adjoining roots of a frozen variable can get rid of floor functions in the coefficients of a generalized exchange relation.
\end{remark}

\begin{definition}
    Let $\mathcal{A}(\mathbf{x}, B, \mathcal{P})$ be a generalized cluster algebra over ground ring $\mathbb{Z} \mathbb{P}$. Let $D = \prod_i d_i$ be the total multiplicity of $\mathcal{A}(\mathbf{x}, B, \mathcal{P})$. We define the embedding of $\mathbb{ZP}$ algebras $\widetilde{\tau_D}$ as the composition of every $\varphi_{j,D}$ over each frozen variable $f_j$:
    \begin{align*}   
    \circ_{i=1}^M( \varphi_{N+j,D}) =: \widetilde{\tau_D}: \mathcal{A}(\mathbf{x}, B, \mathcal{P}) \hookrightarrow \mathcal{A}(\mathbf{x}, \widetilde{B}, \widetilde{\mathcal{P}}).
    \end{align*}
    here the new generalized cluster algebra $\mathcal{A}(\widetilde{B},\widetilde{\mathcal{P}})$ has its $\mathbb{ZP}$ algebra structure given by the following diagram,

    \begin{equation*}
    \adjustbox{scale=0.75}{\begin{tikzcd}
	{\mathbb{ZP}} & {\mathcal{A}^g(B,\mathcal{P})} && {\mathcal{A}(\overline{B},\overline{\mathcal{P}})} && \bullet &&&& \bullet && {\mathcal{A}^g(\widetilde{B},\widetilde{\mathcal{P}})}
	\arrow["\iota", from=1-1, to=1-2]
	\arrow["{\varphi_{1,D}}", from=1-2, to=1-4]
	\arrow["{\varphi_{2,D}}", from=1-4, to=1-6]
	\arrow["{…}", from=1-6, to=1-10]
	\arrow["{\varphi_{M,D}}", from=1-10, to=1-12]
\end{tikzcd}}
    \end{equation*}
\end{definition}

\begin{definition}\label{def: generalized coefficients}
    The coefficients $\rho_{kr} \in \widetilde{\mathcal{P}}$ are called \emph{generalized coefficients}. 
\end{definition}

\begin{remark}
By equations \ref{eq: section one relations} and \ref{eq: q monomial generalization} it follows that 
\begin{equation*}
    \rho_{kr} =\widetilde{\tau_D}(p_{kr}) \cdot \widetilde{\tau_{\frac{D}{d_k}}} (\frac{1}{q_{kr}})=(p_{kr})^D (\frac{1}{q_{kr}})^{\frac{D}{d_k}}
\end{equation*}
\end{remark}

\begin{remark}\label{def: generalized roots}
    For each mutable vertex $x_k$ of $\mathcal{A}(\mathbf{x}, \widetilde{B}, \widetilde{\mathcal{P}})$, by the preceding remark \ref{remark: remove floor functions}, the generalized exchange polynomial $\theta_k$ is homogeneous in the variables $u_{k>}v_{k>}^{[1]}$ and $u_{k<} v_{k<}^{[1]}$. Let $\tau_k := \frac{u_{k>} v_{k>}^{[1]}}{u_{k<} v_{k<}^{[1]}}$ then the inhomogeneous polynomial $\eta (\tau_k)$ correponding to $\theta_k$ has the coefficients $\rho_{kr}$, that is, $\eta(\tau_k) = \sum_{r=0}^{d_k} \rho_{kr} \cdot (\tau_{k})^r \in \mathbb{ZP}[\tau_k]$. Since the ground ring $\mathbb{ZP}$ is an integral domain, say $\mathfrak{F}$ is the algebraic closure of the field of fraction $\text{Frac}(\mathbb{ZP})$ then $\eta(\tau_k) = \prod_{r=1}^{d_k} (\tau_k  - \mathfrak{R}^k_r)$ where each $\mathfrak{R}^k_r \in \mathfrak{F}$ is a root of $\eta(\tau_k)\in \mathfrak{F}[\tau_k]$. Now for each $r$, say $s_r,t_r\in \mathfrak{F}$ such that $\mathfrak{R}^k_r = -\frac{t_r}{s_r}$ then we have,
    \begin{equation*}
        \eta(\tau_k)= \prod_{r=1}^{d_k} (s_r  + t_r \cdot \tau_k)
    \end{equation*}
    Thus $\theta_k$ splits into a product of linear homogeneous polynomials when considered as a polynomial over $\mathfrak{F}$:
    \begin{align}\label{eq: generalized roots}
        \theta_k = \left(u_{k<} v_{k<}^{[1]}\right)^{d_k} \left( \prod_{r=1}^{d_k} (s_r  + t_r \cdot \tau_k)\right)  = \prod_{r=1}^{d_k} (s_r \cdot u_{k<}v_{k<}^{[1]} + t_r \cdot u_{k>}v_{k>}^{[1]}) \in \mathfrak{F}[x_1,...,\widehat{x_k},...,x_N]
    \end{align}
    These relations between the coefficients $\{s_r,t_r\}_{r=1}^{d_k} \subset \mathfrak{F}$ of the linear homogeneous factors of $\theta_k$, as a polynomial over $\mathfrak{F}$, and the generalized coefficients $\rho_{kr}\in  \mathfrak{F}$ will be used to define an ideal $\mathcal{I}$ later.
\end{remark}
\color{black}

\subsection{Construction of an unfolding $\mathcal{B}$ of $B$}\label{subsection: Transformations}
Consider a generalized cluster algebra $\mathcal{A}^g=\mathcal{A}^g(x,B,\mathcal{P})$ and say $\widehat{D}$ is the diagonalizer of the extended-exchange matrix $B$ and $\mathcal{D}$ is the psuedo-rank. In this section, we consider the extended-exchange matrix $B$ of a generalized seed and construct a new extended-exchange matrix $\mathcal{B}$ of size $\mathcal{D} \times (3\mathcal{D}+M)$. See the figure \ref{fig: matrix construction good figure} following this for a depiction of this matrix for a special case where $\mathcal{B}$ corresponds to a folded quiver $\mathcal{Q}$; the construction for the general case is similar. 

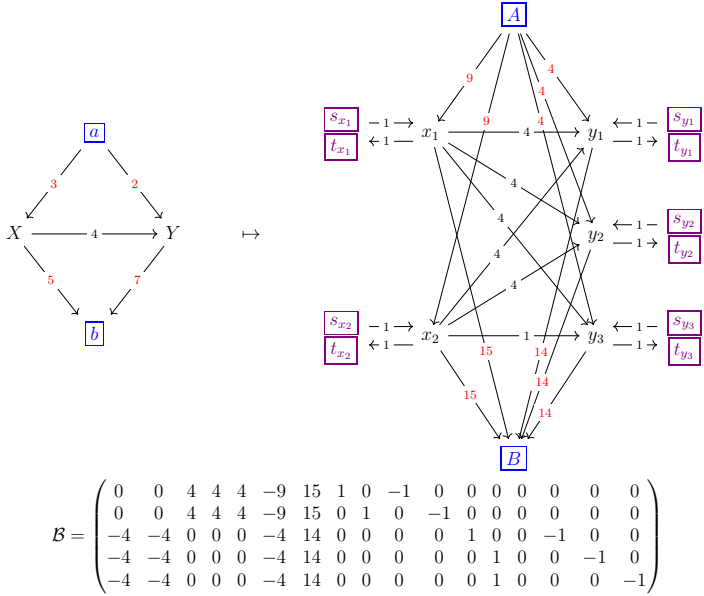
\begin{figure}
    \centering
    \begin{subfigure}[b]{0.98\textwidth}
         \centering
        \begin{minipage}[c]{0.63\textwidth}
        \adjustbox{scale=0.9,center}{$B = \left(\begin{array}{cc|cc}  
 0 & 8 & -3 & 5 \\
-12 & 0  &-2 & 7 
\end{array}\right)$ and $\widehat{B} = \left(\begin{array}{cc|cc}  
 0 & 4 &-3 & 5 \\
 -4 & 0  &-2 & 7  
\end{array}\right)$}
  \end{minipage}\hfill
  \begin{minipage}[c]{0.37\textwidth}
    \caption{An example of an exchange matrix $B$ associated with a generalized seed. Here $\widehat{B}$ is the modified exchange matrix associated with the choice of divisors $D_X=2$ and $D_Y=3$ of the rows of $B$}
  \end{minipage}
     \end{subfigure}
     \begin{subfigure}[b]{0.98\textwidth}
         \centering
        \begin{minipage}[c]{0.61\textwidth}
    \adjustbox{scale=0.6,center}{%
        \begin{tikzcd}
	&&&&&& {\color{blue}\boxed{A}} \\
	\\
	& {\color{blue}\boxed{a}} &&& {{\color{violet}\boxed{s_{x_1}}\atop \color{violet}\boxed{t_{x_1}}}} & {x_1} && {y_1} & {{\color{violet}\boxed{s_{y_1}}\atop \color{violet}\boxed{t_{y_1}}}} \\
	X && Y & \mapsto &&&& {y_2} & {{\color{violet}\boxed{s_{y_2}}\atop \color{violet}\boxed{t_{y_2}}}} \\
	& {\color{blue}\boxed{b}} &&& {{\color{violet}\boxed{s_{x_2}}\atop \color{violet}\boxed{t_{x_2}}}} & {x_2} && {y_3} & {{\color{violet}\boxed{s_{y_3}}\atop \color{violet}\boxed{t_{y_3}}}} \\
	\\
	&&&&&& {\color{blue}\boxed{B}}
	\arrow["\color{red}9"{description}, from=1-7, to=3-6]
	\arrow["\color{red}4"{description, pos=0.4}, from=1-7, to=3-8]
	\arrow["\color{red}4"{description, pos=0.3}, from=1-7, to=4-8]
	\arrow["\color{red}9"{description, pos=0.3}, from=1-7, to=5-6]
	\arrow["\color{red}4"{description, pos=0.3}, from=1-7, to=5-8]
	\arrow["\color{red}3"{description}, from=3-2, to=4-1]
	\arrow["\color{red}2"{description}, from=3-2, to=4-3]
	\arrow["1"{description, pos=0.4}, shift left=2, from=3-5, to=3-6]
	\arrow["1"{description, pos=0.6}, shift left=2, from=3-6, to=3-5]
	\arrow["4"{description, pos=0.6}, from=3-6, to=3-8]
	\arrow["4"{description}, from=3-6, to=4-8]
	\arrow["4"{description, pos=0.4}, from=3-6, to=5-8]
	\arrow["\color{red}15"{description, pos=0.7}, from=3-6, to=7-7]
	\arrow["1"{description, pos=0.6}, shift right=2, from=3-8, to=3-9]
	\arrow["\color{red}14"{description, pos=0.7}, from=3-8, to=7-7]
	\arrow["1"{description, pos=0.4}, shift right=2, from=3-9, to=3-8]
	\arrow["4"{description}, from=4-1, to=4-3]
	\arrow["\color{red}5"{description}, from=4-1, to=5-2]
	\arrow["\color{red}7"{description}, from=4-3, to=5-2]
	\arrow["1"{description, pos=0.6}, shift right=2, from=4-8, to=4-9]
	\arrow["\color{red}14"{description, pos=0.7}, from=4-8, to=7-7]
	\arrow["1"{description, pos=0.4}, shift right=2, from=4-9, to=4-8]
	\arrow["1"{description, pos=0.4}, shift left=2, from=5-5, to=5-6]
	\arrow["4"{description, pos=0.4}, from=5-6, to=3-8]
	\arrow["4"{description}, from=5-6, to=4-8]
	\arrow["1"{description, pos=0.6}, shift left=2, from=5-6, to=5-5]
	\arrow["1"{description, pos=0.6}, from=5-6, to=5-8]
	\arrow["\color{red}15"{description}, from=5-6, to=7-7]
	\arrow["1"{description, pos=0.6}, shift right=2, from=5-8, to=5-9]
	\arrow["\color{red}14"{description, pos=0.7}, from=5-8, to=7-7]
	\arrow["1"{description, pos=0.4}, shift right=2, from=5-9, to=5-8]
\end{tikzcd}
        }
  \end{minipage}
        \hfill
         \begin{minipage}[c]{0.37\textwidth}
    \caption{To the left is the quiver $Q$ associated to $\widehat{B}$. To the right is the folded quiver $\mathcal{Q}$. The blue vertices indicate the original frozen variables and what they are transformed into, the violent vertices indicate the frozen vertices introduced in the last step of our construction, and the red arrows indicate arrows between mutatable and frozen vertices whose number is multiplied by special factors determined by the total multiplicity $\mathcal{D}$}
  \end{minipage}
     \end{subfigure}
    \begin{subfigure}[b]{0.98\textwidth}
        \centering
        \begin{minipage}[c]{0.61\textwidth}
    \adjustbox{scale=0.6,center}{$\mathcal{B} = \begin{pmatrix}
            0 & 0 & 4 & 4 & 4 & -9 & 15 & 1 & 0 & -1 & 0 & 0 & 0 & 0 & 0 & 0 & 0 \\
            0 & 0 & 4 & 4 & 4 & -9 & 15 & 0 & 1 & 0 & -1 & 0 & 0 & 0 & 0 & 0 & 0 \\
            -4 & -4 & 0 & 0 & 0 & -4 & 14 & 0 & 0 & 0 & 0 & 1 & 0 & 0 & -1 & 0 & 0 \\
            -4 & -4 & 0 & 0 & 0 & -4 & 14 & 0 & 0 & 0 & 0 & 0 & 1 & 0 & 0 & -1 & 0 \\
            -4 & -4 & 0 & 0 & 0 & -4 & 14 & 0 & 0 & 0 & 0 & 0 & 1 & 0 & 0 & 0 & -1
        \end{pmatrix}
            $}
  \end{minipage}\hfill
  \begin{minipage}[c]{0.37\textwidth}
    \caption{Here $\mathcal{B}$ is the exchange matrix constructed out of $B$}
  \end{minipage}   
    \end{subfigure}
     \caption{A special case of our construction of $\mathcal{B}$ when it has an  associated quiver $\mathcal{Q}$.}
    \label{fig: matrix construction good figure}
     \captionsetup{width=.98\linewidth}
\end{figure} 

We will specify a sub-matrices of $\mathcal{B}$ as triples $(\mathcal{B},[m_1, m_2], [n_1,n_2])$ where the intervals $[m_1,m_2]$ and $[n_1,n_2]$ denotes the underlying coordinates of the sub-matrix of $\mathcal{B}$. We will decompose $\mathcal{B}$ into three sub-matrices denoted by 
\begin{equation}\label{eq: submatrix 1}
    \hbox{Principal} (\mathcal{B}): =(\mathcal{B},[1,\mathcal{D}] ,[1,\mathcal{D}])
\end{equation}
and 
\begin{equation}\label{eq: submatrix 2}
    \mathcal{B}_S: = (\mathcal{B},[1,\mathcal{D}],[\mathcal{D}+1, \mathcal{D}+M])
\end{equation}
and 
\begin{equation}\label{eq: submatrix 3}
    \mathcal{B}_I := (\mathcal{B},[1,\mathcal{D}],[\mathcal{D}+M+1,3\mathcal{D}+M])
\end{equation}
so that $\hbox{Principal}(\mathcal{B})$ is an unfolding of $\hbox{Principal}(B)$ in the sense of \ref{def: unfolding}. We will prove that this unfolding relationship holds through any sequence of composite mutations. These three sub-matrices will be constructed out of double constant blocks of sub-matrices, and we prove that this double constant block structure holds through sequences of composite mutations. Even more,  $\mathcal{B}_S$ will have the same size as a grid of blocks as $\hbox{Slack}(\mathcal{B})$ does as a matrix, and similarly $\hbox{Principal}(\mathcal{B})$ will have the same size as a grid of blocks as $\hbox{Principal}(B)$ does as a matrix. Both block matrices will be constructed out of constant blocks. 

The values of these constant blocks of $\mathcal{B}_S$ will be determined up to a consistent factor by the values of the entries of $\hbox{Slack}(B)$ through any sequence of composite mutations and corresponding mutations. Likewise, the values of the constant blocks of $\hbox{Principal}(\mathcal{B})$ will be proved to be determined up to a consistent factor by the values of the entries of $\hbox{Principal}(B)$ through any sequence of composite mutations and corresponding mutations.

\subsubsection{Construction of $\mathcal{B}$} \label{subsubsection: unfold}
\label{subsubsection: exchange matrix model}

Let $\mathbb{0}^{N\times M}$ and $\mathbb{1}^{N\times M}$ denote the $N \times M$  integer matrices with a single value across all their entries given by $0$ and $1$, respectively, and let $I_n$ denote the  $n\times n $ identity matrix.

\begin{enumerate}
    \item Let $\mathcal{B}$ be a $\mathcal{D} \times (3\mathcal{D}+M)$ matrix with integer entries so that $\mathcal{B}\in \mathbb{Z}^{\mathcal{D} \times (3\mathcal{D}+M)}$. Consider the sub-matrices from \ref{eq: submatrix 1}, \ref{eq: submatrix 2}, and \ref{eq: submatrix 3}. 
    
    Let ${}^P \mathcal{B}$ be the block matrix whose underlying matrix is $\hbox{Principal}(\mathcal{B})$ and which is an $N\times N$ grid of blocks where the $(i,j)$ block, denoted as $({}^P\mathcal{B})^{i,j}$ is a $d_i\times d_j$ matrix. Similarly, let ${}^S \mathcal{B}$ be the block matrix whose underlying matrix is $ \mathcal{B}_S$ and which is an $N\times N$ grid of blocks where the $(i,j)$ block is an $d_i\times 1$ matrix. Finally, let $I$ be a block matrix whose underlying matrix is $\mathcal{B}_I$ and which is an $N \times N$ grid of blocks where the $(i,j)$ block is an $d_i \times 2d_j$ matrix. In addition, let the block $I^{i,j}$ be itself a $1\times 2$ grid of blocks where the $(1,1)$ block is denoted as ${}^1 I^{i,j}$ and the $(1,2)$ block is denoted by ${}^2 I^{i,j}$, so that in this notation,
    \begin{equation*}
        I^{i,j}= [{}^1 I^{i,j}, {}^2 I^{i,j}]
    \end{equation*}
    and we refer to these as the \textit{Components} of the block $I^{i,j}$.

    Throughout this work, we will reference these various groups of columns associated with $\mathcal{B}$, and now we denote these groups of columns for later use. Let $\mathcal{D}^i$ denote the group of columns underlying the $i^{th}$ block column of ${}^P\mathcal{B}$ and $F$ denote the group of columns underlying ${}^S \mathcal{B}$. Let $W^i$ denote the index set of the columns underlying the $i^{th}$ block column of $I$,  $T^i$ denote the index set of the underlying columns of the first components ${}^1 I^{-,i}$ and $S^i$ denote the index set of the underlying columns of the second components ${}^2 I^{-,i}$ so that essentially $W^i = T^i \sqcup S^i$ for all $i$.
    
    \item Say $i,j\in \{1,...,N\}$ then let $({}^P \mathcal{B})^{i,j}$ be the constant block with constant entry $\frac{1}{d_i} B_{i,j}$, that is,
    \begin{equation}\label{eq: hadamard factor 1}
        ({}^P \mathcal{B})^{i,j} = \frac{1}{d_i} (\hbox{Principal}(B))_{i,j} \mathbbm{1}^{d_i\times d_j}
    \end{equation}
    Note that this assigns a constant value to every block of ${}^P \mathcal{B}$
    
    \item Say $i \in \{1,...,N\}$ and $j = \{1,...,M\} $ then let $({}^S \mathcal{B})^{i,j}$ be the constant block with constant entry $\frac{D}{d_i} B_{i,M+j}$, that is,
    \begin{equation}\label{eq: hadamard factor 2}
        ({}^S \mathcal{B})^{i,j} = \frac{D}{d_i} (\hbox{Slack}(B))_{i,j} \mathbbm{1}^{d_i\times 1}
    \end{equation}
    Note that this assigns a constant value to every block of ${}^S \mathcal{B}$
    
    \item Say $i,j \in \{1,...,N\}$ if $i\not = j$, then let $I^{i,j}$ be the constant block given by the constant value 0 across all entries. Otherwise, if $i=j$, then let ${}^1(I^{i,i})$ be the diagonal matrix  ${}^1(I^{i,i})=1 \cdot I_{d_i} \in \mathbb{Z}^{d_i\times d_i}$, similarly, let ${}^2(I^{i,i}) =-1 \cdot I_{d_i} \in  \mathbb{Z}^{d_i\times d_i}$ . The steps above fully describe the matrix $I$. 
\end{enumerate}

Since these steps uniquely determine a value for each $\mathcal{B}$ entry, these steps specify the matrix $\mathcal{B}$.

\subsubsection{Group mutation class of $\mathcal{B}$}\label{subsubsection:group mutation class}

In this section, we establish a correspondence between the mutation class of $B$ and the group mutation class of $\mathcal{B}$. In what follows, for each natural number $n$, let $[n]:=\{1,..,n\}$

Notice that ${}^P \mathcal{B}$ is a square grid of blocks with zero blocks on the diagonal since $B$  skew-symmetrizable, that is,
\begin{equation}\label{eq: valid group mutation}
    ({}^P \mathcal{B})^{i,i} = \mathbb{0}^{d_i \times d_i}
\end{equation}

It is a basic fact of matrix mutations that whenever an entry of the exchange matrix is zero, such as $B_{i,j}=0$, then the corresponding mutations $\mu_i$ and $\mu_j$ commute on the exchange matrix $B$ that is 
\begin{equation}\label{eq: commute condition}
    (\mu_i \circ \mu_j ) (B) = (\mu_j \circ \mu_i ) (B)
\end{equation}
There is a simple explanation of the statement above in the case that the principal part of $B$ is skew-symmetric (so that a quiver $Q=(Q_0,Q_1)$ may represent $B$). If $B_{i,j}=0$ then there are no arrows between the associated vertices $v_i,v_j\in  Q_0$, and hence the mutation $\mu_i$ does not change the set of vertices which are incident to the vertex $v_j$. This implies that $\mu_i$ and $\mu_j$ commute since a quiver mutation at a vertex is entirely determined by the edges incident to said vertex. 

Given the size correspondence between ${}^P \mathcal{B}$ and $\hbox{Principal}(B)$, for any $j \in \{1,...,N\}$ we consider the group mutation $(\widehat{\mu_{j}})(\mathcal{B}):= (\circ_{ j'\in \mathcal{D}^j} \mu_{j'}) (\mathcal{B})$

By (\ref{eq: commute condition}) and (\ref{eq: valid group mutation}), it follows that each of the mutations that comprise the group mutation commutes with the other. For any finite sequence $s=(s_1,..,s_n)\in [N]^n$ of length $n\in \mathbb{N}$ we have a unique associated mutation of $B$:
\begin{equation}\label{def: mutation sequence}
    \mu_{s} (B) := (\mu_{s_n}\circ...\circ \mu_{s_1})(B)
\end{equation}
and a unique associated sequence of group mutations of $\mathcal{B}$,
\begin{equation}\label{def: group mutation sequence}
    \widehat{\mu_{s}} (\mathcal{B}) := (\widehat{\mu_{s_n}}\circ...\circ \widehat{\mu_{s_1}})(\mathcal{B})
\end{equation}
Thus, we have a one to one correspondence between matrices mutation equivalent to $B$ and matrices, which are group mutation equivalent to $\mathcal{B}$ by sending $B$ to $\mathcal{B}$ and also sending $\mu_{s} (B)$ to $\widehat{\mu_{s}} (\mathcal{B})$ for all $s\in [N]^n$ and $n\in \mathbb{N}$.

\tiny
\begin{figure}[b]
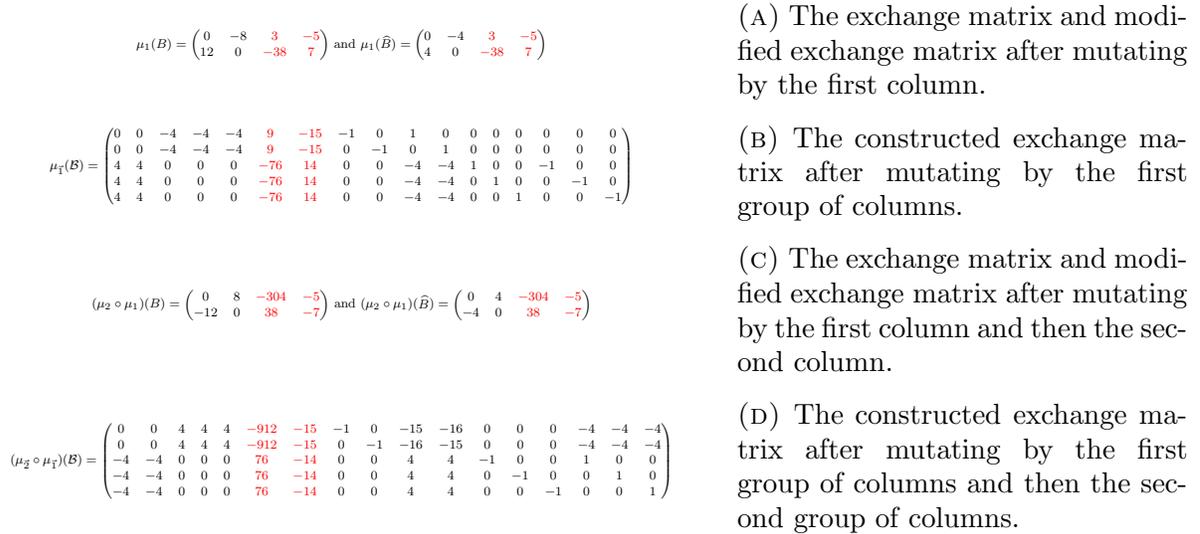

    \centering
    \begin{subfigure}[b]{0.98\textwidth}
        \centering

        \begin{minipage}[c]{0.61\textwidth}
    \adjustbox{scale=0.6,center}{$\mu_{1}(B)=\begin{pmatrix}
        0 & -8 & \color{red}3 & \color{red}-5\\
        12 & 0 & \color{red}-38 & \color{red}7
        \end{pmatrix}$ and $\mu_{1}(\widehat{B})=\begin{pmatrix}
        0 & -4 & \color{red}3 & \color{red}-5\\
        4 & 0 & \color{red}-38 & \color{red}7
        \end{pmatrix}$}
  \end{minipage}\hfill
  \begin{minipage}[c]{0.37\textwidth}
    \caption{The exchange matrix and modified exchange matrix after mutating by the first column.}
  \end{minipage}
    \end{subfigure}

    \begin{subfigure}[b]{0.98\textwidth}
        \centering

        \begin{minipage}[c]{0.61\textwidth}
    \adjustbox{scale=0.6,center}{$ \mu_{\vec{1}}(\mathcal{B})=\begin{pmatrix}
        0 & 0 & -4 & -4 & -4 & \color{red}9 & \color{red}-15 & -1 & 0 & 1 & 0 & 0 & 0 & 0 & 0 & 0 & 0\\
        0 & 0 & -4 & -4 & -4 & \color{red}9 & \color{red}-15 & 0 & -1 & 0 & 1 & 0 & 0 & 0 & 0 & 0 & 0\\
        4 & 4 & 0 & 0 & 0 & \color{red}-76 & \color{red}14 & 0 & 0 & -4 & -4 & 1 & 0 & 0 & -1 & 0 & 0 \\ 
        4 & 4 & 0 & 0 & 0 & \color{red}-76 & \color{red}14 & 0 & 0 & -4 & -4 & 0 & 1 & 0 & 0 & -1 & 0\\
        4 & 4 & 0 & 0 & 0 & \color{red}-76 & \color{red}14 & 0 & 0 & -4 & -4 & 0 & 0 & 1 & 0 & 0 & -1
        \end{pmatrix}$}
  \end{minipage}\hfill
  \begin{minipage}[c]{0.37\textwidth}
    \caption{The constructed exchange matrix after mutating by the first group of columns.}
  \end{minipage}
    \end{subfigure}

    \begin{subfigure}[b]{0.98\textwidth}
        \centering

        \begin{minipage}[c]{0.61\textwidth}
    \adjustbox{scale=0.6,center}{$(\mu_2 \circ \mu_1)(B)=\begin{pmatrix}
        0 & 8 & \color{red}-304 & \color{red}-5\\
        -12 & 0 & \color{red}38 & \color{red}-7
        \end{pmatrix}$ and $(\mu_2 \circ \mu_1)(\widehat{B})=\begin{pmatrix}
        0 & 4 & \color{red}-304 & \color{red}-5\\
        -4 & 0 & \color{red}38 & \color{red}-7
        \end{pmatrix}$}
  \end{minipage}\hfill
  \begin{minipage}[c]{0.37\textwidth}
    \caption{The exchange matrix and modified exchange matrix after mutating by the first column and then the second column. }
  \end{minipage}
    \end{subfigure}

    \begin{subfigure}[b]{0.98\textwidth}
        \centering

        \begin{minipage}[c]{0.61\textwidth}
    \adjustbox{scale=0.6,center}{$(\mu_{\vec{2}}\circ \mu_{\vec{1}})(\mathcal{B})=\begin{pmatrix}
        0 & 0 & 4 & 4 & 4 & \color{red}-912 & \color{red}-15 & -1 & 0 & -15 & -16 & 0 & 0 & 0 & -4 & -4 & -4\\
        0 & 0 & 4 & 4 & 4 & \color{red}-912 & \color{red}-15 & 0 & -1 & -16 & -15 & 0 & 0 & 0 & -4 & -4 & -4\\
        -4 & -4 & 0 & 0 & 0 & \color{red}76 & \color{red}-14 & 0 & 0 & 4 & 4 & -1 & 0 & 0 & 1 & 0 & 0 \\ 
        -4 & -4 & 0 & 0 & 0 & \color{red}76 & \color{red}-14 & 0 & 0 & 4 & 4 & 0 & -1 & 0 & 0 & 1 & 0\\
        -4 & -4 & 0 & 0 & 0 & \color{red}76 & \color{red}-14 & 0 & 0 & 4 & 4 & 0 & 0 & -1 & 0 & 0 & 1
        \end{pmatrix}$}
  \end{minipage}\hfill
  \begin{minipage}[c]{0.37\textwidth}
    \caption{The constructed exchange matrix after mutating by the first group of columns and then the second group of columns. }
  \end{minipage}
         
    \end{subfigure}
    
    \caption{The exchange matrices from figure (\ref{fig: matrix construction good figure}) after a sequence of mutations and the corresponding sequence of group mutations. Here the non-mutatable or slack columns of the exchange matrix $B$ and its associated sub-matrix of $\mathcal{B}$ are red.}
    \label{fig: Group mutation example}
\end{figure}
\normalsize

We will now prove two key lemmas which will describe how the blocks of $\mathcal{B}$ change under group mutations. See figure (\ref{fig: Group mutation example}) for an example of how $\mathcal{B}$ may change under group muations.

By the construction in section \ref{subsection: Transformations}, the constant value across any entry of a given block of ${}^P \mathcal{B}$ is given by a corresponding entry of $\hbox{principal}(B)$ times some special factors, as in \ref{eq: hadamard factor 1}, and likewise the constant value across any block of ${}^S \mathcal{B}$ is given by the value of a corresponding entry of $\hbox{Slack}(B)$ times some special factor, as in \ref{eq: hadamard factor 2}. Using the correspondence \ref{def: group mutation sequence} we will prove that these relationships hold between any sequence of mutations of $B$ and the associated group mutation sequence of $\mathcal{B}$.

\begin{lemma}[Hadamard Conditions]\label{lemma: group mutation rule}
 Let $\ell \in \{1,...,M\}$, $i,j\in \{ 1,...,N\}$ with $i\not = j$, and $s=(s_1,...,s_n)$ with $1\leq s_i \leq N$. Consider the sequence of mutations $\mu_s$  defined on $B$ by $s$ as in \ref{def: mutation sequence} and the associated sequence of group mutations $\widehat{\mu_s}$ on $\mathcal{B}$ as in \ref{def: group mutation sequence}. Then the following conditions hold for the blocks of ${}^P\mathcal{B}$ and ${}^S\mathcal{B}$,
 \begin{equation}\label{eq: Hadamard Condition 1}
     (\widehat{\mu_{s}} ({}^P \mathcal{B}))^{i,j}= \frac{1}{d_i}\mu_{s} (B)_{i,j} \mathbbm{1}^{d_i \times d_j}=   \frac{1}{d_i}\mu_{s} (\hbox{Principal}(B))_{i,j} \mathbbm{1}^{d_i \times d_j}
 \end{equation}
and 
 \begin{equation}\label{eq: Hadamard Condition 2}
      (\widehat{\mu_{s}} ({}^S \mathcal{B}))^{i,\ell } = \frac{D}{d_i}\mu_{s} (B)_{i,N+\ell}\mathbbm{1}^{d_i \times 1}=\frac{D}{d_i}\mu_s(\hbox{Slack}(B))_{i,\ell}\mathbbm{1}^{d_i \times 1}
 \end{equation}
\end{lemma}
We call \ref{eq: Hadamard Condition 1} and \ref{eq: Hadamard Condition 2} the Hadamard conditions $1$ and $2$ respectively since they describe persistent factors of the values of each block of ${}^P\mathcal{B}$ and ${}^S \mathcal{B}$ through sequences of group mutations, which reminded us of a Hadamard product between matrices. For instance, the theorem says that given a sequence $\widehat{\mu_s}$ of group mutations of $\mathcal{B}$, we can determine the mutated matrices $\widehat{\mu_s}({}^P\mathcal{B})$ and $\widehat{\mu_s}({}^S\mathcal{B})$ by instead considering 
\begin{equation*}
    \mu_s (B)\odot H 
\end{equation*}
where $\odot$ denotes the Hadamard product between matrices and $H$ is a matrix of the same size given by 
\begin{equation*}
    H=\begin{bmatrix}
        \frac{1}{d_1} & \frac{1}{d_1} & ... & \frac{1}{d_1} & \frac{D}{d_1} & ... &\frac{D}{d_1}\\
        \frac{1}{d_2} & \frac{1}{d_2} & ... & \frac{1}{d_2} & \frac{D}{d_2}& ... &\frac{D}{d_2}\\
        ...  & ... & ... & ... & ... & ... & ...\\
        \frac{1}{d_N} & \frac{1}{d_N} & ... & \frac{1}{d_N} & \frac{D}{d_N}& ... &\frac{D}{d_N} 
    \end{bmatrix}
\end{equation*}

\begin{proof}
It suffices to show that for any sequence $s\in [N]^n$ of length $n\in \mathbb{Z}_{\geq 0}$ the lemma holds. We will prove the lemma by using the method of mathematical induction. Both of these conditions are met between the initial extended exchange matrix $B$ and $\mathcal{B}$ by the construction in \ref{subsubsection: exchange matrix model}. We will show that these conditions hold after a single mutation and omit the inductive step as it will follow by the same proof.

Note that $\mathbb{1}^{d_i \times d_k} \cdot \mathbb{1}^{d_k \times d_j} = d_k \cdot \mathbb{1}^{d_i \times d_j}.$ and consider the group mutation $\mu_k$ for some $k\in [N]$ and also consider a block ${}^P\mathcal{B}^{i,j}$. Say $\mathbb{1}:=\mathbb{1}^{d_i \times d_j}$

By \ref{eq : grouping_mutation} if $k=i$ or $k=j$ then condition \ref{eq: Hadamard Condition 1} holds. In addition, if $k\not = i$ and $k\not = j$ then
\begin{gather*}
    \mu_k({}^P \mathcal{B}^{i,j})={}^P \mathcal{B}^{i,j}+ \frac{1}{2} \left[ \hbox{sgn}({}^P\mathcal{B}^{i,k}){}^P\mathcal{B}^{i,k}{}^P\mathcal{B}^{k,j}+\hbox{sgn}({}^P\mathcal{B}^{k,j}){}^P\mathcal{B}^{i,k}{}^P\mathcal{B}^{k,j} \right]\\
    = \frac{1}{d_i} B_{i,j} \mathbb{1}+\frac{1}{2}\left[ d_k\hbox{sgn}(B_{i,k})\frac{1}{d_i}B_{i,k} \frac{1}{d_k}B_{k,j} \mathbb{1}+d_k \hbox{sgn}(B_{k,j})\frac{1}{d_i}B_{i,k}\frac{1}{d_k}B_{k,j} \mathbb{1} \right]  =\frac{1}{d_i} \mu_k(B) \mathbb{1}
\end{gather*}
and so condition \ref{eq: Hadamard Condition 1} holds after a single mutation.

By the construction in \ref{subsubsection: exchange matrix model} each block of ${}^S \mathcal {B}$ is a constant matrix and hence has a well defined sign value in $\{-1,0,1\}$ which is the common sign value  across all its entries. . Thus, a similar equation to \ref{eq : grouping_mutation} holds for the blocks ${}^S B ^{i,\ell}$, in particular,
\begin{align} 
        \widehat{\mu_k}\left( {}^S\mathcal{B}^{i,\ell}  \right) = \begin{cases} - {}^S\mathcal{B}^{i,\ell}  & \text{$k = i$} \\ {}^S\mathcal{B}^{i,\ell} + \frac{1}{2}\left(\text{sgn}({}^P\mathcal{B}^{i,k})({}^P\mathcal{B}^{i,k})({}^S\mathcal{B}^{k,\ell} )+ \text{sgn}\left({}^S\mathcal{B}^{k,\ell}\right)({}^P\mathcal{B}^{i,k})({}^S\mathcal{B}^{k,\ell} )\right) & \hbox{otherwise.} \end{cases}
\end{align}

Say $\mathbb{J}:=\mathbb{1}^{d_i\times 1}$ and note that $\mathbb{1}^{d_i \times d_k} \cdot \mathbb{1}^{d_k \times 1} = d_k \mathbb{J} $. Now if $k=i$ or $k=\ell$ the condition \ref{eq: Hadamard Condition 2} holds. On the other hand, if $k\not = i$ and $k\not = \ell$ then,
\begin{gather*}
    \widehat{\mu_k}\left( {}^S\mathcal{B}^{i,\ell}  \right)  =  \frac{D}{d_i} B_{i,N+\ell} \mathbb{J}+\frac{1}{2}\left[ d_k\hbox{sgn}(B_{i,k})\frac{1}{d_i}B_{i,k} \frac{D}{d_k} B_{k,N+\ell} \mathbb{J}+d_k \hbox{sgn}(B_{k,j})\frac{1}{d_i}B_{i,k} \frac{D}{d_k} B_{k,N+\ell} \mathbb{J} \right] \\
    = \frac{D}{d_i} \mu_k ( B)_{i,N+\ell} \mathbb{J}
\end{gather*}
and so condition \ref{eq: Hadamard Condition 2} holds in this case as well.

\end{proof}
\color{black}

\begin{remark}
    Condition \ref{eq: Hadamard Condition 1} implies that the $\hbox{Principal}(\mathcal{B})$ is an unfolding of $\hbox{Principal}(B)$ in the sense of definition \ref{def: unfolding}.
\end{remark}

By the construction of \ref{subsubsection: exchange matrix model}, the components of each block $I^{i,j}$ are double constant matrices. In particular, if $i\not = j$ so that $I^{i,j}$ is an off-diagonal block of $I$, then $I^{i,j}$ is genuinely a constant matrix in the sense it's off-diagonal value and its diagonal value are the same. We now prove that this double constant structure on the blocks of $I$ holds through any sequence of group mutations.

\begin{lemma}[Double Constant Condition] \label{lem: double_const_cond}
    Consider the sub-matrix
\begin{equation*}
    I = (\mathcal{B},[1,\mathcal{D}],[\mathcal{D}+M+1,3\mathcal{D}+M]) 
\end{equation*} as a block matrix as described in \ref{subsection: Transformations}. The following relations hold and are preserved via group mutations.

For each pair of blocks $({}^1 I^{i,j}, {}^2 I^{i,j})$, there exist some $a_{ij}, c_{ij} \in \mathbb{Z}$ such that 
\begin{align*}
    &{}^1 I^{i,j} = - {}^2 I^{i,j} + a_{ij} \cdot \mathbb{1}^{d_i \times d_j}, \\[2ex]
    &{}^1 I^{i,j} = \begin{cases}
        c_{ij} \cdot \mathbb{1}^{d_i \times d_j} & \text{if } i \neq j \\
        c_{ii} \cdot \mathbb{1}^{d_i \times d_i} \pm I_{d_i} & \text{if } i = j
    \end{cases}
\end{align*}
\end{lemma}

\begin{proof}
    It suffices to show that for any sequence $s\in [N]^n$ of length $n\in \mathbb{Z}_{\geq 0}$ the lemma holds. We will prove the lemma by using the method of mathematical induction. Both of these conditions are met between the initial blocks ${}^1 I^{i,j}$ and ${}^2 I^{i,j}$ by the construction in \ref{subsubsection: exchange matrix model} where all  $a_{i,j}=0$ and all $c_{i,j}=0$. We will show that these conditions hold after a single group mutation and omit the inductive step as it will follow by the same proof.

These relations imply that each block ${}^1 I^{i,j}$ has either all non positive entries, or all non negative entries, or all zero entries. Thus, we can always consider the sign of one of these blocks, which may take values $\{-1,0,1\}$. Therefore, given a group mutation $\widehat{\mu_k}$, the block ${}^1 I^{i,j}$ mutates by an equation similar to \ref{eq : grouping_mutation}, in particular,
\begin{align} \label{eq: component one rule}
        \widehat{\mu_k}({}^1 I^{i,j })  = \begin{cases} - {}^1 I^{i,j}  & \text{$k = i$ } \\ {}^1 I^{i,j} + \frac{1}{2}\left(\text{sgn}({}^PB^{i,k}) {}^P\mathcal{B}^{i,k} \cdot  {}^1 I^{k,j} + \text{sgn}\left( {}^1 I^{k,j}\right) {}^P \mathcal{B}^{i,k} \cdot {}^1 I^{k,j}\right) &  \hbox{otherwise}  \end{cases}
\end{align}
and similarly 
\begin{align} \label{eq: component two rule}
        \widehat{\mu_k}({}^2 I^{i,j })  = \begin{cases} - {}^2 I^{i,j}  & \text{$k = i$ or $k=\ell$} \\ {}^2 I^{i,j} + \frac{1}{2}\left(\text{sgn}(B^{i,k})\mathcal{B}^{i,k} \cdot  {}^2 I^{k,j} + \text{sgn}\left({}^2 I^{k,j}\right) \mathcal{B}^{i,k} \cdot {}^2 I^{k,j}\right) & \hbox{otherwise.} \end{cases}
\end{align}

Notice that $\mathbb{1}^{d_i \times d_k} \cdot \mathbb{1}^{d_k \times d_j} = d_k \cdot \mathbb{1}^{d_i \times d_j}$. 

We first show these relations hold between pairs of blocks $({}^1I^{k,j},{}^2I^{k,j})$. By these mutation rules,  $\widehat{\mu}_k\left({}^1 I^{k,k}\right) = -I_{d_k}$ and  $\widehat{\mu}_k\left({}^2 I^{k,k}\right) = I_{d_k}$ and whenever $k\not = j$  since ${}^1 I^{k,j}=0={}^2 I^{k,j}$ we have  $\widehat{\mu}_k\left({}^1 I^{k,j}\right) = 0  \cdot \mathbb{1}^{d_k\times d_j}$ and  $\widehat{\mu}_k\left({}^2 I^{k,j}\right) = 0  \cdot \mathbb{1}^{d_k\times d_j}$ . 

We will now show these relations hold between pairs of blocks $({}^1I^{i,j},{}^2I^{i,j})$ for $i\not =k$. First say $j=k$, then ${}^P \mathcal {B}^{i,k}$ is a constant matrix and ${}^1 I^{k,k}$ is a double constant matrix. Since the product of a constant matrix with a double constant matrix is a constant matrix it follows that the mutation rules add integer multiples of $\mathbb{1}^{d_i \times d_k}$ to ${}^1 I^{i,j}$ and also to ${}^2 I^{i,j}$, hence $\widehat{\mu_k}({}^1 I^{i,i }) =  c_1  \mathbb{1}^{d_i \times d_i} + I_{d_i}$ and $\widehat{\mu_k}({}^2 I^{i,i }) =  c_2  \mathbb{1}^{d_i \times d_i} - I_{d_i}$ for some integers $c_1,c_2 \in \mathbb{Z}$. Then it follows that  $\widehat{\mu_k}({}^1 I^{i,i }) = - \widehat{\mu_k}({}^2 I^{i,i }) +(c_1+c_2)\mathbb{1}^{d_i \times d_j}$ and so the relations hold. Lastly, say $j\not=k$, then ${}^P \mathcal {B}^{i,k}$ is a constant matrix and ${}^1 I^{k,j}$ is a constant matrix so their matrix product is also a constant matrix and it follows that the mutation rule adds integer multiples of $\mathbb{1}^{d_i\times d_j}$ to ${}^1I^{i,j}$ and also to ${}^2 I^{i,j}$. By the previous reasoning, it follows that these relations hold for the pair of blocks $({}^1I^{i,j},{}^2I^{i,j})$ for $i\not =k$. 

We have shown that these relations hold after a single group mutation. 
\end{proof}

\color{black}

\subsection{The Folded Cluster Algebra associated to $\mathcal{B}$}\label{subsection:CCA}
Say $\mathcal{A}^g(\mathbf{x},B,\mathcal{P})$ is a generalized cluster algebra with exchange matrix $B$; consider the associated exchange matrix $\mathcal{B}$ which is compatible with $B$ in the sense of the Hadamard and double constant conditions discussed in  \ref{subsubsection:group mutation class}. In this section, we will consider a (folded) cluster algebra $\mathcal{A}^c$ generated by $\mathcal{B}$, say the ground ring of this cluster algebra is $\mathbb{Z \widehat{P}}$. The ground ring $\mathbb{Z \widehat{P}}$ will have an embedded copy of the ground ring $\mathbb{ZP}$ of $\mathcal{A}^g$ . We will consider an ideal $\mathcal{I} \lhd\mathbb{Z\widehat{P}}$ so that in the quotient $\mathbb{Z\widehat{P}}\diagup \mathcal{I }$ the classes of frozen variables associated to the columns of the submatrix $\mathcal{B}_{I}$ have the same relations to the classes of the generalized coefficients as the relations between the coefficients of the linear homogeneous factors of each generalized exchange polynomial of $\mathcal{A}^g(B)$, as in remark  \ref{def: generalized roots}, and the generalized coefficients. We consider the extension $\mathcal{I}^e$ of the ideal $\mathcal{I}$  to $\mathcal{A}^c$ and for seeds of the cluster algebra $\mathcal{A}^c$, which are group mutation equivalent to the initial seed, we will then prove a formula that expresses the product of the classes of exchange polynomials in $\mathcal{A}^c\diagup \mathcal{I}^e$ for any group of cluster variables in terms of these generalized coefficients. 

Consider the ambient field $\mathcal{F}' =\mathbb{Q}(X_1,...,X_{3\mathcal{D}+M})$ in $3\mathcal{D}+M$ variables. Fix the cluster $\vec{y}=(X_1,...,X_{\mathcal{D}})$ and the ground ring $\mathbb{D}=\mathbb{Z}[(X_{\mathcal{D}+1})^{\pm 1},...,(X_{3\mathcal{D}+M})^{\pm 1}]$. Recall that the exchange matrix $\mathcal{B}$ from \ref{subsection: Transformations} has $3\mathcal{D}+M$ columns, and its principal part was a $\mathcal{D}$ by $\mathcal{D}$ matrix. Consider the (traditional) cluster algebra $\mathcal{A}^c = \mathcal{A}(\vec{y},\mathcal{B})= \mathbb{D}[\mathcal{Y}]\subset \mathcal{F}'$ of rank $\mathcal{D}$ over the ring $\mathbb{D}$ of total multiplicity one, where $\mathcal{Y}$ is the set of cluster variables formed by mutations of the initial seed $(\vec{y},\mathcal{B})$. 

Throughout this section and the rest of this work, we often relabel the indeterminates of $\mathcal{F}'$ to reference the columns of $\mathcal{B}$. In particular, we relabel the ordered set of indeterminates as,
\begin{gather}\label{eq: chain label decomposition}
    \{X_1,...,X_{3\mathcal{D} +M}\} \\= \{ y_{1,1},...,y_{1,d_1}\}\sqcup ...\sqcup\{ y_{N,1},...,y_{N,d_N}\}\sqcup \{f_1,...,f_M\}\\ \sqcup \{t_{1,1},...,t_{1,d_1}\}\sqcup  \{s_{1,1},...,s_{1,d_1}\} \sqcup ... \sqcup \{t_{N,1},...,t_{N,d_N}\}\sqcup  \{s_{N,1},...,s_{N,d_1}\}
\end{gather}
Sometimes, we even consider a relabeling of certain factors of this disjoint union into a new ordered set, such as 
\begin{equation*}
    \{y_1,...,y_{\mathcal{D}}\} =  \{ y_{1,1},...,y_{1,d_1}\}\sqcup ...\sqcup\{ y_{N,1},...,y_{N,d_N}\}
\end{equation*}
or even 
\begin{equation*}
    \{t_1,...,t_{\mathcal{D}}\} = \{t_{1,1},...,t_{1,d_1}\} \sqcup ... \sqcup \{t_{N,1},...,t_{N,d_N}\}
\end{equation*}
or 
\begin{equation*}
    \{s_1,...,s_{\mathcal{D}}\} = \{s_{1,1},...,s_{1,d_1}\} \sqcup ... \sqcup \{s_{N,1},...,s_{N,d_N}\}
\end{equation*}
Recall the sets $\mathcal{D}^j$, $F$, $T^j$, and $S^j$ denoting the sets of columns underlying the different collections of block columns of $\mathcal{B}$ from \ref{subsection: Transformations}. In this notation the columns of $\mathcal{D}^j$ corresponds to the indeterminants $\{y_{j,1},...,y_{j,d_j}\}$, the columns of $F$ correspond to the indeterminants $\{f_1,...,f_M\}$, the columns of $T^j$ correspond to the indeterminants $\{t_{j,1},...,t_{j,d_j}\}$, and the columns of $S^j$ correspond to the indeterminants $\{s_{j,1},...,s_{j,d_j}\}$.

In what follows we will denote the first $n$ natural numbers as $[n]:= \{1,...,n\}$.

For the given generalized cluster algebra $\mathcal{A}^g$ recall the associated generalized coefficients $\{\rho_{k,r}\}$ (definition \ref{def: generalized coefficients}) . Let $\mathcal{I}\lhd\mathbb{Z\widehat{P}}$ be the ideal so that in the quotient the classes of frozen variables associated to the columns of the submatrix $\mathcal{B}_{I}$ have the same relations to the classes of the generalized coefficients as the relations between the coefficients of the linear homogeneous factors of each generalized exchange polynomial of $\mathcal{A}^g(B)$, as in remark  \ref{def: generalized roots}, and the generalized coefficients, more precisely, let $\mathcal{I}$ be generated by elements of the form

\begin{align}\label{eq: main ideal relation}
    \rho_{kr} - \sum_{ \substack{I\cup J = \mathcal{D}^k \\ |I| = r  \\ I\cap J = \emptyset}
} \left(\prod_{i\in I} s_i \cdot \prod_{j \in J} t_j \right).
\end{align}

\begin{remark}
    Consider the field $\mathfrak{F}$ as in remark \ref{def: generalized roots} which is the algebraic closure of the field of fractions of $\mathbb{ZP}$ and also consider the roots $\mathfrak{R}^k_{r}$ of the exchange polynomials. The subring 
    \begin{equation*}
        \mathbb{ZP}[\mathfrak{R}]:=\mathbb{ZP}[\mathfrak{R}_1^1,...,\mathfrak{R}_{d_1}^1,...,\mathfrak{R}_1^n,...,\mathfrak{R}_{d_N}^n]
    \end{equation*}
    of $\mathfrak{F}$ obtained by adjoining all the roots to $\mathbb{ZP}$ can also be considered as a subring of the quotient $\mathbb{Z\widehat{P}}\diagup \mathcal{I}$. There is a unique embedding of $\mathbb{Z}$ algebras $\varphi:\mathbb{ZP}[\mathfrak{R}]\to \mathbb{Z\widehat{P}}\diagup \mathcal{I}$ such that $\varphi(f_i)=[F_i]_{\mathcal{I}}$ and $\varphi( \mathfrak{R}_j^k)= [-\frac{t_{k,j}}{s_{k,j}}]_{\mathcal{I}}=[-t_{k,j} \prod_{i\in [d_k]\setminus \{j\}} s_{k,i}]_{\mathcal{I}}$, and this is also a ring embedding. We therefore identify the classes $[-t_{k,j} \prod_{i\in [d_k]\setminus \{j\}} s_{k,i}]_{\mathcal{I}}$ with the associated root $\mathfrak{R}_j^k$. It follows that the classes $[-t_{k,j} \prod_{i\in [d_k]\setminus \{j\}} s_{k,i}]_{\mathcal{I}}$ have the same relations to the generalized coefficients as the relations between the roots and the generalized coefficients, hence,
    \begin{equation*}
        \prod_{r=1}^{d_k} [1  + t_{k,r} \prod_{i\in [d_k]\setminus \{r\}} s_{k,i}]_{\mathcal{I}}=\sum_{r=0}^{d_k} [\rho_{kr}]_{\mathcal{I}}
    \end{equation*}
\end{remark}

\color{black}

\subsubsection{Group Mutation Equivalent Seeds} Say $\mathbf{s}_0$ is the initial seed of $\mathcal{A}^c$ and,
\begin{equation*}
    \mathbf{s}= (\widehat{\mu_{n_p}} \circ ... \circ \widehat{\mu_{n_1}}) (\mathbf{s}_0),
\end{equation*}
is a seed obtained by a group mutation sequence as in (\ref{def: group mutation sequence}) and with $\mathcal{B}'$ is the exchange matrix obtained by the action of the corresponding matrix mutation sequence on $\mathcal{B}$, that is $\mathcal{B}'= (\widehat{\mu_{n_p}} \circ ... \circ \widehat{\mu_{n_1}}) (\mathcal{B})$.
Now consider a particular group $i \in \{1,...,N\}$ and say $i',i'' \in \mathcal{D}^i$ so that the associated cluster variables are in the same group at this seed $\mathbf{s}$, that is, $y_{i'}^{\mathbf{s}},y_{i''}^{\mathbf{s}} \in \mathcal{A}^c_{\mathbf{s}'}$. Recall that each cluster variable has an associated set of Laurent monomials and Cluster monomials, which the cluster variables index.

Hadamard condition 1 of section \ref{subsubsection:group mutation class} implies that the cluster monomials $(\mu_{i'>})^{\mathbf{s}}$ and $(\mu_{i''>})^{\mathbf{s}}$ associated to cluster variables in the same group are equal. Likewise, Hadamard condition 2 implies that the Laurent monomials of cluster variables in the same group have a particular common divisor. The following lemma will make this precise.

\begin{remark}[Abuse of Notation]\label{remark: Abuse of notation}
    In the following lemma, by the Hadamard conditions \ref{eq: Hadamard Condition 1} and \ref{eq: Hadamard Condition 2}, we will abuse the notation where we use the symbol ${}^P \mathcal{B}^{i,k}$ to refer to the constant value across all the entries of this sub-matrix, even though we have defined ${}^P \mathcal{B}^{i,k}$ to be a sub-matrix. This notation is evident from the context and has a well-defined meaning as ${}^P \mathcal{B}^{i,k}$ is a constant sub-matrix and will remain constant through group mutations via the Hadamard conditions.
\end{remark}

\begin{lemma}\label{corollary: group monomials}
Let $\mathbf{s}$ be a seed that is group mutation equivalent to the initial seed $\mathbf{s}_0$, with $\mathcal{B}'$ being this seed's corresponding exchange matrix. Say $i\in \{ 1,...,N\}$ and $i' \in \mathcal{D}^i$ so that  $y_{i'}^{\mathbf{s}}$ is the associated cluster variable in the $i^{th}$ group. Then the following holds, 
\begin{equation*}
    (u_{i'>})^{\mathbf{s}} = \prod_{\substack{ 1\leq k\leq N \\ ({}^P \mathcal{B}')^{i,k} > 0}} (\prod_{j\in \mathcal{D}^i} y_j)^{({}^P \mathcal{B}')^{i,k}} \hbox{ and } (u_{i'<})^{\mathbf{s}} = \prod_{\substack{ 1\leq k\leq N \\ ({}^P \mathcal{B}')^{i,k} < 0}} (\prod_{j\in \mathcal{D}^i} y_j)^{({}^P \mathcal{B}')^{i,k}}
\end{equation*}
\textit{In other words, two cluster variables in the same group have the same cluster monomials}

In addition, the stable monomials corresponding to a given pair of cluster variables in the same group, such as $(v_{i'>})^{\mathbf{s}}$ and $(v_{i''>})^{\mathbf{s}}$, share a common divisor given by
\begin{equation*}
     \prod_{\substack{m\in [M]\\ ({}^S \mathcal{B}')^{i,m} >0}} (f_{m})^{({}^S \mathcal{B}')^{i,m} } 
\end{equation*}
and similarly the stable monomials $(v_{i'<})^{\mathbf{s}}$ and $(v_{i''<})^{\mathbf{s}}$ share a common divisor given by
\begin{equation*}
    \prod_{\substack{m\in [1 , M  ]\\ ({}^S \mathcal{B}')^{i,m} <0}} (f_{m})^{({}^S \mathcal{B}')^{i,m} } 
\end{equation*}
and this holds for all $i', i''\in \mathcal{D}^i$.
\end{lemma}
\begin{proof}
We will prove the lemma for a single composite/group mutation, and the proof for a general finite sequence of mutations will follow by a similar argument. In particular, say $\mathcal{B}' = \widehat{\mu_{k}} (\mathcal{B})$ and $\mathbf{s}=\widehat{\mu_k}(\mathbf{s}_0)$.

We first consider the part of this lemma relating to the cluster monomials. Let $i \in [N]$ and $i' \in \mathcal{D}^i$, the associated cluster monomials to $i'$ is given by
\begin{equation*}
    u_{i'>}:= \prod_{\substack{  k\in [\mathcal{D}] \\ \mathcal{B}_{i',k}' > 0}}y_{k}^{\mathcal{B}_{i',k}'} = \prod_{k\in [N]}\prod_{\substack{k'\in \mathcal{D}^k \\ \mathcal{B}_{i',k'}' > 0}}y_{k'}^{\mathcal{B}_{i',k'}'}
\end{equation*}
By the Hadamard condition \ref{eq: Hadamard Condition 1}, whenever $k',k'' \in \mathcal{D}^k$ then we have $\mathcal{B}_{i',k'}' = \mathcal{B}_{i',k''}'$. Keeping in mind that the block $({}^P\mathcal{B}')^{i,k}$ has a constant value across all its entries and using this constant value, for all $k' \in \mathcal{D}^k$ we have the following expression,
\begin{equation*}
    \mathcal{B}_{i',k'}' = ({}^P\mathcal{B}')^{i,k}
\end{equation*}
we make sense of this statement using the constant value across all the entries of  $({}^P\mathcal{B}')^{i,k}$. Now, in this simplified notation, we may write,
\begin{equation*}
    u_{i'>} = \prod_{\substack{k\in [N]\\ ({}^P \mathcal{B}')^{i,k} > 0}} (\prod_{k'\in \mathcal{D}^k} y_{k'})^{({}^P \mathcal{B}')^{i,k}}
\end{equation*}
where again $({}^P\mathcal{B})^{i,k}$ denotes the constant value across the block guaranteed by the Hadamard condition 1, and so the lemma follows. Note that by the Hadamard condition 1 whenever $i,k\in \{1,...,N\}$ and we consider  $i',i'' \in \mathcal{D}^i$ and $k'\in \mathcal{D}^k$ then,
\begin{equation*}
    \mathcal{B}_{i',k'} = \mathcal{B}_{i'',k'}
\end{equation*}
then it follows by the expression for the cluster monomials that,
\begin{equation*}
    u_{i' >} = \prod_{\substack{  1 \leq k\leq \mathcal{D} \\ \mathcal{B}_{i',k}' > 0}}y_{k}^{\mathcal{B}_{i',k}'}=\prod_{\substack{  1 \leq k\leq \mathcal{D} \\ \mathcal{B}_{i'',k}' > 0}}y_{k}^{\mathcal{B}_{i'',k}'}=u_{y_{i''} >}
\end{equation*}
The case for $u_{i' <}$ is similar.

We now consider the part of this lemma regarding the stable monomials. The stable monomial associated with $i'$ is given by
\begin{gather*}
    v_{i' >}:= \prod_{\substack{ k\in \{\mathcal{D}  +1,...,3\mathcal{D}+M \}\\ \mathcal{B}_{i',k} > 0}}X_{i}^{\mathcal{B}_{i',k}} = (\prod_{\substack{m\in [M]\\ \mathcal{B}_{i',m+\mathcal{D}}>0}} (f_{m})^{\mathcal{B}_{i',m+\mathcal{D}}}) (\prod_{\substack{ k\in  \{\mathcal{D} +M +1 ,..., 3 \mathcal{D}+M\} \\ \mathcal{B}_{i',k} > 0}}(X_k)^{\mathcal{B}_{i',k}})
\end{gather*}
Consider the following factor of this stable monomial,
\begin{equation*}
     \prod_{\substack{m\in [M]\\ \mathcal{B}_{i',m+\mathcal{D}}>0}} (f_{m})^{\mathcal{B}_{i',m+\mathcal{D}}}
\end{equation*}
Since $i'\in \mathcal{D}^i$ and $(m+\mathcal{D})\in [\mathcal{D}+1,\mathcal{D}+M]$, it follows the $(i',m+\mathcal{D})$ entry of $\mathcal{B}'$ is an entry of the constant block $({}^S \mathcal{B}')^{i,m}$. Then, by the Hadamard condition $2$, it follows $\mathcal{B}_{i',m+\mathcal{D}}' =  ({}^S \mathcal{B}')^{i,m} $. Now we have that
\begin{equation*}
    \prod_{\substack{m\in [M]\\ \mathcal{B}_{i',m+\mathcal{D}}>0}} (f_{m})^{\mathcal{B}_{i',m+\mathcal{D}}}=\prod_{\substack{m\in [M]\\ ({}^S \mathcal{B}')^{i,m} >0}} (f_{m})^{({}^S \mathcal{B}')^{i',m} } 
\end{equation*}
On the other hand, another consequence of the Hadamard condition 2 is that if $i',i''\in \mathcal{D}$, then $\mathcal{B}_{i',m}' = \mathcal{B}_{i'',m}'$ hence, it immediately follows that,
\begin{equation*}
    \prod_{\substack{m\in [M]\\ ({}^S \mathcal{B}')^{i,m} >0}} (f_{m})^{({}^S \mathcal{B}')^{i',m} } =\prod_{\substack{m\in [M]\\ ({}^S \mathcal{B}')^{i,m} >0}} (f_{m})^{({}^S \mathcal{B}')^{i'',m} } 
\end{equation*}
This proves the lemma part of the lemma regarding stable monomials.
\end{proof}

For seeds $\mathbf{s}$, which are group mutation equivalent to the initial seed $\mathbf{s}_0$, we now define these cluster monomials in the $\mathbb{D}$ sub-algebra generated by the cluster variables of $\mathbf{s}$, which we have previously denoted in the preliminaries by  $\mathcal{A}^c_{\mathbf{s}}$.
\begin{definition} \label{def : group monomials}
Let $\mathbf{s}$ be a seed which is group mutation equivalent to the initial seed $\mathbf{s}_0$, in particular, $\mathbf{s}= (\widehat{\mu_{n_o}} \circ ... \circ \widehat{\mu_{n_1}}) (\mathbf{s}_0)$ with $\mathcal{B}'$ being the corresponding exchange matrix of this seed. Recall that the constant values of the blocks of ${}^P\mathcal{B}'$ and ${}^S\mathcal{B}'$ are determined by the entries of the matrix $B'=(\mu_{n_o} \circ ... \circ \mu_{n_1}) (B)$ and our notation as in remark \ref{remark: Abuse of notation}.

For $1\leq k \leq N$, the following monomials are well defined by lemma (\ref{corollary: group monomials}), and we now take these as definitions and denote these by
\begin{equation*}
    (U_{k>})^{\mathbf{s}} =  \prod_{\substack{ i\in [1,N]\\ ({}^P \mathcal{B}')^{k,i} > 0}} (\prod_{i'\in \mathcal{D}^i} y_{i'})^{({}^P \mathcal{B}')^{k,i}} \quad \quad (U_{k<})^{\mathbf{s}} =  \prod_{\substack{ i\in [1,N]\\ ({}^P \mathcal{B}')^{k,i} < 0}} (\prod_{i'\in \mathcal{D}^i} y_{i'})^{({}^P \mathcal{B}')^{k,i}}
\end{equation*}
and
\begin{equation*}
    (V_{k>})^{\mathbf{s}} =\prod_{\substack{m\in [1 , M  ]\\ ({}^S \mathcal{B}')^{k,m} >0}} (f_{m})^{({}^S \mathcal{B}')^{k,m} } \quad \quad (V_{k<})^{\mathbf{s}} =\prod_{\substack{m\in [1 , M  ]\\ ({}^S \mathcal{B}')^{k,m} <0}} (f_{m})^{({}^S \mathcal{B}')^{k,m} }
\end{equation*}

\end{definition}

\begin{lemma}[Product Formula]\label{rem: product_formula}
    In the quotient $\mathcal{A}^c/\mathcal{I}^e$, we can compute the product of the classes of all cluster variables in a group $k$ is given as
\begin{align*}
    \left[ \prod_{k' \in \mathcal{D}^k} \theta_{k'} \right]_{\mathcal{I}^e}
    = \left[\sum_{r=0}^{d_k} \rho_{kr} \cdot U_{k>}^r V_{k>}^r U_{k<}^{d_k-r} V_{k<}^{d_k-r} \right]_{\mathcal{I}^e}. 
\end{align*}
\end{lemma}

\begin{proof}
    We will compute the exponents of the $\{s_{k,k'}\}$ and $\{t_{k,k'}\}$ coefficients of $\theta_{k'}^{\mathbf{t}}$ in an arbitrary seed $\mathbf{t}$ using the double constant condition. The blocks $[^1I^{k,k}]^{\mathbf{t}} = b_k \mathbb{1}^{d_k \times d_k} + \alpha_kI_{d_k}$ and $[^2I^{k,k}]^{\mathbf{t}} = c_k \mathbb{1}^{d_k \times d_k} - \alpha_kI_{d_k}$ at the seed $\mathbf{t}$ give us the contributions of the variables $\{s_{k,k'}\}$ and $\{t_{k,k'}\}$ to $\theta_{k'}^{\mathbf{t}}$, respectively. Notice that by the relations $\prod_{k' \in \mathcal{D}^k} s_{k,k'} = \prod_{k' \in \mathcal{D}^k} t_{k,k'} = 1$ in $\mathcal{I}$, we can ignore all blocks $I^{i,j}$ except when $i = j$, because their contribution to each exchange polynomials, as factors of each term, are equal to one in this quotient.  

    This formula holds at the initial seed $\mathbf{s}_0$. We can prove this statement with mathematical induction by first assuming that the formula holds at an arbitrary seed $\mathbf{s}$ then proving that the condition holds at any adjacent seed $\mathbf{s}'=\mu_\ell(\mathbf{s})$.

    If $k=\ell$ then the theorem follows by the mutation rule and our assumption that the product formula holds at the seed $\mathbf{t}$.

    Say $k\not =\ell$, without loss of generality, by the double constant condition assume $\mu_{\mathbf{s}}({}^1I^{k,k})=c_1\mathbb{1}^{d_k\times d_k}+I_k$ for some integer $c_1$ which implies $\mu_{\mathbf{s}}({}^2I^{k,k})=c_2\mathbb{1}^{d_k\times d_k}-I_k$ for some integer $c_2$. This implies that the class of the coefficient $p_{kr}$ is given by the class $\left[p_{kr}\right]_{\mathcal{I}^e}=\prod_{k'\in \mathcal {D}^k}\left[ s_{k'}+t_{k'}\right]_{\mathcal{I}^e}=\prod_{k'\in \mathcal{D}^k}\left[\frac{v_{k'>}}{V_{k>}} +\frac{v_{k'<}}{V_{k<} }\right]_{\mathcal{I}^e}^{\mathbf{s}}$ By the mutation rules of \ref{eq: component one rule} and \ref{eq: component two rule} it follows that there will exists integers $e_1$ and $e_2$ so that $\mu_{\mathbf{s}'}({}^1I^{k,k})=e_1\mathbb{1}^{d_k\times d_k}+I_k$ and $\mu_{\mathbf{s}'}({}^2I^{k,k})=e_2\mathbb{1}^{d_k\times d_k}-I_k$. This implies that the class of the coefficient $p_{kr}$ is given by the class $\left[p_{kr}\right]_{\mathcal{I}^e}=\prod_{k'\in \mathcal {D}^k}\left[ s_{k'}+t_{k'}\right]_{\mathcal{I}^e}=\prod_{k'\in \mathcal{D}^k}\left[\frac{v_{k'>}}{V_{k>}} +\frac{v_{k'<}}{V_{k<} }\right]_{\mathcal{I}^e}^{\mathbf{s}'}$
    
    Hence the product formula holds for seed $\mathbf{s}'$.
    
\end{proof}

\subsection{The Embedding}\label{sec: embedding}
    Let $\mathcal{A}^g(\mathbf{x}, B, \mathcal{P})$ be any generalized cluster algebra(GCA), $\mathcal{A}^g(\mathbf{x}, \widetilde{B}, \widetilde{\mathcal{P}})$ be the associated GCA defined in section \ref{section: Ring of Coefficients}, and $\mathcal{A}^c(\widehat{\mathbf{x}}, \mathcal{B})$ be the associated cluster algebra defined in \ref{subsection:CCA}. Each seed $\mathbf{s}$ in $\mathcal{A}^g(\mathbf{x}, \widetilde{B}, \widetilde{\mathcal{P}})$ has a corresponding seed $\widehat{\mathbf{s}}$ in $\mathcal{A}^c(\widehat{\mathbf{x}}, \mathcal{B})$ obtained by identifying each mutation $\mu_k$ with the corresponding group mutation $\widehat{\mu}_k$, similarly there is a seed $\mathbf{t}$ of $\mathcal{A}^g(B,\mathcal{P})$ which corresponds to $\mathbf{s}$. Also, for each cluster variable $X_{k}^{(\mathbf{s})}$ in $\mathbf{s}$, we have exactly $d_k$ corresponding cluster variables $X_{k,1}^{(\widehat{\mathbf{s}})}, \cdots, X_{k,d_k}^{(\widehat{\mathbf{s}})}$ in $\widehat{\mathbf{s}}$. Each frozen variable $f_k$ in the initial seed $\mathbf{s}_0$ corresponds to a frozen variable in the initial seed $\widehat{\mathbf{s}}_0$, which we will also denote by $F_k$ in the following theorem.

\begin{theorem}
    There exists a unique embedding of $\mathbb{Z}$ algebras
    \begin{align*}
        &\Phi : \mathcal{A}^g(\mathbf{x}, \widetilde{B}, \widetilde{\mathcal{P}}) \hookrightarrow \mathcal{A}^c(\widehat{\mathbf{x}}, \mathcal{B})/\mathcal{I}
    \end{align*}
    such that for each cluster variable $x_k^{\mathbf{s}}$ and frozen variable $f_j$ of $\mathcal{A}^g$ the map is defined by $\Phi\left(x_k^{\mathbf{s}}\right)= \left[X_{k,1}^{(\widehat{\mathbf{s}})} \cdots X_{k,d_k}^{(\widehat{\mathbf{s}})} \right]_{\mathcal{I}^e}$ and $\Phi(f_j) = F_j$.
\end{theorem}

\begin{proof}
    Say $\mathbb{Z\widehat{P}}[\widehat{\mathbf{s}_0}^{\pm 1}]$ denotes ring of Laurent polynomials in the cluster variables of the seed $\widehat{s_0}$ with coefficients in $\mathbb{Z\widehat{P}}$ and $\mathcal{I}^E$ denote the ideal extension of $\mathcal{I}\lhd \mathbb{Z\widehat{P}}$ into $\mathbb{Z\widehat{P}}[\widehat{\mathbf{s}_0}^{\pm 1}]$. By the Laurent phenomena of cluster algebras, there exists a unique $\mathbb{Z}$-algebra embedding
    \begin{equation*}
        \Phi:\mathcal{A}^g(\mathbf{x},\widetilde{B},\widetilde{P}) \to \mathbb{Z\widehat{P}}[\widehat{\mathbf{s}_0}^{\pm 1}] \diagup \mathcal{I}^E
    \end{equation*}
    such that $\Phi(x_k^{\mathbf{s}_0}) = \left[X_{k,1}^{(\widehat{\mathbf{s}_0})} \cdots X_{k,d_k}^{(\widehat{\mathbf{s}_0})} \right]_{\mathcal{I}^E}$ and $\Phi(f_j)= F_j$ for all $k\in [N]$ and $j\in [M]$. 
    
    Let $\mathcal{Y}$ be the set of cluster variables of the cluster algebra $\mathcal{A}^c$, $\mathbb{Z\widehat{P}}[\mathcal{Y}]$ be the $\mathbb{Z\widehat{P}}$ subalgebra of $\mathbb{Z\widehat{P}}[\widehat{\mathbf{s}_0}^{\pm 1}] \diagup \mathcal{I}^E$ generated by the classes of each cluster variable of $\mathcal{Y}$, and $\pi: \mathbb{Z\widehat{P}}[\widehat{\mathbf{s}_0}^{\pm 1}]\to \mathbb{Z\widehat{P}}[\widehat{\mathbf{s}_0}^{\pm 1}]\diagup \mathcal{I}^E$ be the ring quotient map. Notice that $\pi$ restricts and descends to an isomorphism of rings $\pi:\mathcal{A}^c \diagup \mathcal{I}\to \mathbb{Z\widehat{P}}[\mathcal{Y}]$ and, even more, this map is also an isomorphism of $\mathbb{Z}$-algebras. Let $\widetilde{X}$ denote the set of cluster variables of $\mathcal{A}^g(\mathbf{x},\widetilde{B},\widetilde{P})$, since $\Phi(f_j)\in \mathbb{Z\widehat{P}}[\mathcal{Y}]$ for any frozen variable by construction, if 
    \begin{equation}\label{eq: thm condition}
        \Phi(\widetilde{X})\subset \mathbb{ZP}[\mathcal{Y}]
    \end{equation}
    then it follows $\hbox{image}(\Phi)\subset \mathbb{Z\widehat{P}}[\mathcal{Y}]$ so, by restricting the codomain, $\Phi$ is a $\mathbb{Z}$ algebra embedding of $\mathcal{A}^g(\mathbf{x},\widetilde{B},\widetilde{P})$ into $\mathcal{A}^c\diagup \mathcal{I}^e$.

    Let $\mathbf{s}$ be any (generalized) seed which is mutation equivalent to $\mathbf{s}_0$ and recall that for any (generalized) cluster variable $x_k^{\mathbf{s}}$ of $\mathcal{A}^g(\mathbf{x},\widetilde{B},\widetilde{P})$ the corresponding (generalized) exchange polynomial $\theta^{\mathbf{s}}_k$ is homogeneous in the variables $u_{k>}v_{k>}^{[1]}$ and $u_{k<}v_{k<}^{[1]}$ by remarks \ref{remark: remove floor functions} and \ref{def: generalized roots}. In addition, recall the monomials associated to any group of cluster variables of a seed $\widehat{\mathbf{s}}$, which is group mutation equivalent to $\widehat{\mathbf{s}_0}$, defined in definition \ref{def : group monomials}.

    To prove that condition \ref{eq: thm condition} holds we will show that for any generalized seed $\mathbf{s}$ which is mutation equivalent to $\mathbf{s}_0$, and for any indices $k$ corresponding to cluster variables, the following conditions hold
    \begin{enumerate}
        \item[(i)] $\Phi\left(u_{k>}^{\mathbf{s}}  \right) = \left[U_{k>}^{\widehat{\mathbf{s}}} \right]_{\mathcal{I}^E} \quad \hbox{ and } \quad \Phi\left(u_{k<}^{\mathbf{s}}  \right) = \left[U_{k<}^{\widehat{\mathbf{s}}} \right]_{\mathcal{I}^E}$,
        \item[(ii)] $\Phi\left(\left(v_{k>}^{[1]}\right)^{(\mathbf{s})}  \right) = \left[ \left(V_{k>}\right)^{(\widehat{\mathbf{s}})}\right]_{\mathcal{I}^E} \quad \hbox{ and } \quad \Phi\left(\left(v_{k<}^{[1]}\right)^{(\mathbf{s})}  \right) = \left[ \left(V_{k<}\right)^{(\widehat{\mathbf{s}})}\right]_{\mathcal{I}^E}$.
        \item[(iii)] 
        $\varphi(X_k^{\mathbf{s}})=\left[X_{k,1}^{\widehat{\mathbf{s}}} \cdots X_{k,d_k}^{\widehat{\mathbf{s}}} \right]_{\mathcal{I}^E}$
        \item[(iv)]
        $\varphi(p_{kr}^{\mathbf{s}})=\sum_{ \substack{I\cup J = \mathcal{D}^k \\ |I| = r  \\ I\cap J = \emptyset}
} \left[\prod_{k'\in I} \frac{v_{k'<}}{V_{k<}} \cdot \prod_{k'' \in J} \frac{v_{k''>}}{V_{k>}} \right]^{\widehat{\mathbf{s}}}_{\mathcal{I}^e}$
        
    \end{enumerate}
    We can prove that these conditions hold for any arbitrary seed $\mathbf{s}$ which is group mutation equivalent to $\mathbf{s}_0$ by using mathematical induction. These three conditions holds for the initial seed $\mathbf{s}_0$. We will now assume these three conditions hold for an arbitrary seed $\mathbf{s}$ and show that these conditions hold for any adjacent seed $\mathbf{s}' = \mu_{\ell}(\mathbf{s})$. 
    
    Since $\Phi$ fixes any generalized coefficient, that is $\Phi(\rho_{kr}) = \rho_{kr}$, conditions $(i)$ and $(ii)$ holding at seed $\mathbf{s}$ implies that $\Phi(\theta_k^{\mathbf{s}}) = \left[ \theta_{k,1}^{\widehat{\mathbf{s}} }  \cdot ... \cdot \theta_{k,d_k}^{\widehat{\mathbf{s}} }\right]_{\mathcal{I}^E}$ by the product formula \ref{rem: product_formula}, and even more that condition $(iii)$ holds for the seed $\mathbf{s}'$. It suffices to prove conditions $(i)$, $(ii)$, and $(iv)$ hold for $\mathbf{s}'$. 

    Say $k=\ell$, at the seed $\mathbf{s}'$ the conditions $(i)$ and $(ii)$ by the mutation rule and our assumption that conditions $(i)$ and $(ii)$ hold at seed $\mathbf{s}$. Similarly condition $(iv)$ holds in this case by the mutation rules.

    Say $k\not =\ell$ and $\mathfrak{B}$ denotes the modified exchange matrices corresponding to $\widetilde{B}$ and assume $\widetilde{B}_{k,i}>0$ with $i\in [N]$. By the inductive hypothesis, condition $(iii)$ holds for the seed $\mathbf{s}'$. Consider the highest power of the cluster variable $x_i^{\mathbf{s}'}$ that divides the cluster monomial $u_{k>}^{\mathbf{s}'}$, by the hadamard condition \ref{eq: Hadamard Condition 1}, we have
    \begin{gather*}
        \Phi \left( (x_i^{\mathbf{s}'})^{ \mu_{\mathbf{s}'}(\mathfrak{B})_{k,i}} \right)   = \left ( \Phi(x_i^{\mathbf{s}'}) \right)^{\frac{1} {d_k}\mu_{\mathbf{t}'}(B)_{k,i}} = \left[(X_{i,1}^{\widehat{\mathbf{s}}
    '}\cdot ...\cdot X_{i,d_i}^{\widehat{\mathbf{s}}
    '})^{ \left(\mu_{\widehat{\mathbf{s}}
    '}({}^P \mathcal{B}) \right)^{k,i}}\right]_{\mathcal{I}^E}
    \end{gather*}
    which implies that $\Phi(u_{k>}^{\mathbf{s}'}) = U_{k>}^{\widehat{s}'}$. By similar reasoning, it follows that $\Phi(u_{k<}^{\mathbf{s}'}) = U_{k<}^{\widehat{s}'}$ so condition $(i)$ holds for $\mathbf{s}'$.

    On the other hand, say $f_j$ is a frozen variable of $\mathcal{A}(\mathbf{x},\widetilde{B},\widetilde{P})$ with $j\in [M]$, assume $\widetilde{B}_{k,N+j}>0$ and consider the Laurent monomial $\left[f_j^{[1]}\right]^{\mathbf{s}'}$ as in \ref{eq: frozen box}; This is the highest power of the frozen variable $f_j$ that divides the stable monomial $v_{k>}^{\mathbf{s}'}$, by hadamard condition \ref{eq: Hadamard Condition 2} it follows that
    \begin{gather*}
        \Phi\left(   (f_j^{[1]})^{\mathbf{s}'} \right) = \Phi \left( f_j^{\lfloor\frac{\mu_{\mathbf{s}'}(\widetilde{B})_{k,N+j}}{d_j} \rfloor} \right) =\Phi \left( f_j^{\lfloor D\frac{\mu_{\mathbf{t}'}(B)_{k,N+j}}{d_j} \rfloor} \right) = \Phi \left( f_j^{} \right)^{\frac{D}{d_j}\mu_{\mathbf{t}'}(B)_{k,N+j} } =\Phi \left( f_j^{} \right)^{\left(\mu_{\widehat{\mathbf{s}}'}({}^SB) \right)^{k,j} } 
    \end{gather*}
    and this implies that $\Phi\left((v_{k>}^{[1]})^{\mathbf{s}'} \right)=V_{k>}^{\widehat{s}'}$ and by similar reasoning it follows that $\Phi\left((v_{k<}^{[1]})^{\mathbf{s}'} \right)=V_{k<}^{\widehat{s}'}$, so condition $(ii)$ holds for $\mathbf{s}'$.

    Note that whenever $k\not =\ell$ we have $p_{kr}^{\mathbf{s}'}=p_{kr}^{\mathbf{s}}$. By similar reasoning as in the proof of the product formula \ref{rem: product_formula} we may assume $\mu_{\widehat{s}}({}^1I^{k,k})=c_1 \mathbb{1}^{d_k\times d_k}+ I_k$ for some non negative integer $c_1\in \mathbb{Z}_{\geq 0}$. By the double constant condition this implies that the class of the Laurent monomials of our discussion have the following forms, 
    \begin{equation*}
        \left[\frac{v_{k'<}}{V_{k<}}\right]_{\mathcal{I}^e}^{\widehat{\mathbf{s}}}=[s_{k'}]_{\mathcal{I}^e} \quad \hbox{ and } \quad \left[\frac{v_{k''>}}{V_{k>}}\right]_{\mathcal{I}^e}^{\widehat{\mathbf{s}}}=[t_{k''}]_{\mathcal{I}^e}
    \end{equation*}
    Now since $k\not =\ell $ it follows that $\mu_{\widehat{s}'}({}^1I^{k,k})=c_2 \mathbb{1}^{d_k\times d_k}+ I_k$ for some (possible negative) integer $c_2\in \mathbb{Z}$. It follows that the monomials in question at the seed $\widehat{\mathbf{s}}$ are equal to one of the following forms,
    \begin{enumerate}
        \item $\left[\frac{v_{k'<}}{V_{k<}}\right]_{\mathcal{I}^e}^{\widehat{\mathbf{s}}'}=[s_{k'}]_{\mathcal{I}^e} \quad \hbox{ and } \quad \left[\frac{v_{k'>}}{V_{k>}}\right]_{\mathcal{I}^e}^{\widehat{\mathbf{s}}'}=[t_{k'}]_{\mathcal{I}^e}$
        \item $\left[\frac{v_{k'<}}{V_{k<}}\right]_{\mathcal{I}^e}^{\widehat{\mathbf{s}}'}=[s_{k'}t_{k'}^{-1}]_{\mathcal{I}^e} \quad \hbox{ and } \quad \left[\frac{v_{k'>}}{V_{k>}}\right]_{\mathcal{I}^e}^{\widehat{\mathbf{s}}'}=[1]_{\mathcal{I}^e}$
        \item $\left[\frac{v_{k'<}}{V_{k<}}\right]_{\mathcal{I}^e}^{\widehat{\mathbf{s}}'}=[1]_{\mathcal{I}^e} \quad \hbox{ and } \quad \left[\frac{v_{k'>}}{V_{k>}}\right]_{\mathcal{I}^e}^{\widehat{\mathbf{s}}'}=[t_{k'}s_{k'}^{-1}]_{\mathcal{I}^e}$
        \item $\left[\frac{v_{k'<}}{V_{k<}}\right]_{\mathcal{I}^e}^{\widehat{\mathbf{s}}'}=[t_{k'}^{-1}]_{\mathcal{I}^e} \quad \hbox{ and } \quad \left[\frac{v_{k'>}}{V_{k>}}\right]_{\mathcal{I}^e}^{\widehat{\mathbf{s}}'}=[s_{k'}^{-1}]_{\mathcal{I}^e}$
    \end{enumerate}
    In any case, by the relations defining $\mathcal{I}$ it follows that for disjoint subsets $I,J\subset \mathcal{D}^k$ such that $I\cup J=\mathcal{D}^k$ and $|I|=r$ we have
    \begin{equation*}
        \left[\prod_{k'\in I} \frac{v_{k'<}}{V_{k<}} \cdot \prod_{k'' \in J} \frac{v_{k''>}}{V_{k>}} \right]^{\widehat{\mathbf{s}}'}_{\mathcal {I}^e}=\left[\prod_{k'\in I} \frac{v_{k'<}}{V_{k<}} \cdot \prod_{k'' \in J} \frac{v_{k''>}}{V_{k>}} \right]^{\widehat{\mathbf{s}}}_{\mathcal {I}^e}
    \end{equation*}
    and even more that condition $(iv)$ holds at seed $\widehat{s}'$.
    
\end{proof}

\begin{proof}[Proof of Theorem $1$]
    We now prove that any generalized cluster algebra $\mathcal{A}^g(\mathbf{x},B,\mathcal{P})$ over a ground ring $\mathbb{ZP}$ is a subquotient as a $\mathbb{ZP}$ algebra of a (traditional) cluster algebra (with restricted scalars). 

    The folded cluster algebra $\mathcal{A}^c(\widehat{\mathbf{s}_0})$ is a $\mathbb{Z\widehat{P}}$ algebra. For a given group $k$ of cluster variables of $\mathcal{A}^c$ and any given seed $\widehat{\mathbf{s}}$ which is group mutation equivalent to $\widehat{\mathbf{s}_0}$  consider the product of all the cluster variables in this group, namely,
    \begin{equation}\label{eq: cluster products}
X_{k,1}^{\widehat{\mathbf{s}}}\cdot...\cdot X_{k,d_k}^{\widehat{\mathbf{s}}}\in \mathcal{A}^c
    \end{equation}
    Let $\mathcal{A}$ be the subring of $\mathcal{A}^c$ given by forgetting the $\mathbb{ZP}$ algebra structure of the $\mathbb{ZP}$ subalgebra of $\mathcal{A}^c$ generated by elements of the form \ref{eq: cluster products} and $\mathcal{I}^{\mathfrak{E}}$ be the ideal extension of $\mathcal{I}$ into $\mathcal{A}$.

    There is a unique $\mathbb{Z}$ algebra embedding $\eta_D :\mathbb{ZP}\to \mathcal{A}^c(\mathbf{x},\mathcal{B})$ such that $\eta_D( f_j) =F_j^D$. Consider $\mathcal{A}^c$ as a $\mathbb{ZP}$ algebra via the map $\eta_D$. Now the ring quotient $\mathcal{A}\diagup \mathcal{I}^{\mathfrak{E}}$ can be consider as a $\mathbb{ZP}$ algebra subquotient of the cluster algebra $\mathcal{A}^c$. 

    The map $\Phi\circ \widetilde{\tau_D}:\mathcal{A}^g(B,\mathcal{P})\to \mathcal{A}^c(\mathcal{B})\diagup\mathcal{I}^e$ is an embedding of $\mathbb{Z}$ algebra, and when its domain and codomain have the $\mathbb{ZP}$ algebra structures given by the following diagram
    \begin{equation}
        \begin{tikzcd}
	{\mathbb{ZP}} && {\mathcal{A}^g(B,\mathcal{P})} && {\mathcal{A}^g(\widetilde{B},\widetilde{\mathcal{P}})} && {\mathcal{A}^c(\mathcal{B})\diagup\mathcal{I}^e}
	\arrow["\iota", from=1-1, to=1-3]
	\arrow["{\widetilde{\tau_D}}", from=1-3, to=1-5]
	\arrow["\Phi", tail, from=1-5, to=1-7]
\end{tikzcd}
    \end{equation}
    then $\Phi\circ \widetilde{\tau_D}$ is a $\mathbb{ZP}$ algebra embedding. Let $\pi: \mathcal{A}^c\to \mathcal{A}^c\diagup \mathcal{I}^e$ be the ring quotient map given by the ideal $\mathcal{I}^e$. The map $\pi$ is a $\mathbb{ZP}$ algebra homomorphism when its domain and codomain have the previously described $\mathbb{ZP}$ algebra structures. Furthermore, $\pi$ restricts and descends to an embedding of $\mathbb{ZP}$ algebra $\widetilde{\pi}:\mathcal{A}\diagup \mathcal{I}^{\mathfrak{E}}\to \mathcal{A}^c\diagup \mathcal{I}^e$. On the other hand $\hbox{image}(\Phi)=\hbox{image}(\widetilde{\pi})$, and by identifying $\hbox{image}(\widetilde{\pi})$ with $\mathcal{A}\diagup \mathcal{I}^{\mathfrak{E}}$ theorem follows.

\end{proof}

\begin{remark}
    This construction still works if we replace the total multiplicity everywhere with the least common multiple of the degrees $d_1,..,d_N$ of the generalized exchange polynomials of the given generalized cluster algebra $\mathcal{A}^g(B,\mathcal{P})$
\end{remark}

\color{black}

\bibliographystyle{plain}

\end{document}